\pdfoutput=1
\RequirePackage{ifpdf}
\ifpdf 
\documentclass[pdftex]{sigma}
\else
\documentclass{sigma}
\fi

\usepackage{mathrsfs}
\usepackage{bm}

\numberwithin{equation}{section}

\newtheorem{Theorem}{Theorem}[section]
\newtheorem*{Theorem*}{Theorem}
\newtheorem{Corollary}[Theorem]{Corollary}
\newtheorem{Lemma}[Theorem]{Lemma}
\newtheorem{Proposition}[Theorem]{Proposition}
 { \theoremstyle{definition}
\newtheorem{Definition}[Theorem]{Definition}

\newtheorem{Remark}[Theorem]{Remark} }

\newcommand{\leaderfill}{\leaders\hbox to 1em{\hss.\hss}\hfill}
\newcommand{\RR}{\mathbb{R}}

\newcommand{\CC}{\mathbb{C}}

\newcommand{\ZZ}{\mathbb{Z}}
\newcommand{\NN}{\mathbb{N}}

\newcommand{\A}{{\mathscr A}}

\newcommand{\aw}{\mathrm{AW}\!\left(3\right)}
\newcommand{\awc}[2]{\mathrm{AW}_{#1}\!\left(#2\right)}
\newcommand{\AW}[1]{\mathrm{AW}_{#1}}
\newcommand{\awr}{\mathscr{A}\!\!\mathscr{W}}
\newcommand{\bivariateawalgebra}{rank 2 Askey--Wilson algebra }

\newcommand{\chq}{\mathrm{cosh}_q}

\newcommand{\dyone}{\Delta\left(Y_{L}\right)}
\newcommand{\dyzero}{\Delta\left(Y_{K}\right)}

\newcommand{\inprod}[2]{\left\langle #1, #2 \right\rangle}
\newcommand{\half}{\frac{1}{2}}

\newcommand{\ls}{\lambda_s}

\newcommand{\psimn}[2]{\psi_{#1}^{#2}\otimes \psi_{#2}}

\newcommand{\qhyp}[5]{\,_{#1}\phi_{#2} \left[ \begin{matrix}
		#3 \\ #4 \end{matrix};#5 \right]}

\newcommand{\shq}{\mathrm{sinh}_q}
\newcommand{\slc}{\mathfrak{sl}(2,\CC)}

\makeatletter
\newcommand{\subalign}[1]{%
	\vcenter{%
		\Let@ \restore@math@cr \default@tag
		\baselineskip\fontdimen10 \scriptfont\tw@
		\advance\baselineskip\fontdimen12 \scriptfont\tw@
		\lineskip\thr@@\fontdimen8 \scriptfont\thr@@
		\lineskiplimit\lineskip
		\ialign{\hfil$\m@th\scriptstyle##$&$\m@th\scriptstyle{}##$\hfil\crcr
			#1\crcr
		}%
	}%
}
\makeatother

\newcommand{\tensor}[1]{^{\otimes{#1}}}

\newcommand{\uq}{\mathcal{U}_q(\mathfrak{su}(2))}
\newcommand{\Uq}{\mathcal{U}_q}
\newcommand{\UqO}[1][q]{\mathcal{U}_{#1}\left(\Omega_0\right)}
\newcommand{\uqeen}{\mathcal{U}_q(\mathfrak{su}(1,\!1))}
\newcommand{\uql}[1][q]{\mathcal{U}_{#1}(\mathfrak{sl}(2,\CC))}

\newcommand{\vspan}[1]{\text{span}\{#1\}}
\newcommand{\yone}{Y_{L}\otimes 1}
\newcommand{\yzero}{1\otimes Y_{K}}

\begin{document}
\allowdisplaybreaks

\newcommand{\arXivNumber}{2206.03986}

\renewcommand{\PaperNumber}{008}

\FirstPageHeading

\ShortArticleName{An Askey--Wilson Algebra of Rank 2}

\ArticleName{An Askey--Wilson Algebra of Rank 2}

\Author{Wolter GROENEVELT and Carel WAGENAAR}

\AuthorNameForHeading{W.~Groenevelt and C.~Wagenaar}

\Address{Delft Institute of Applied Mathematics, Technische Universiteit Delft,\\
PO Box 5031, 2600 GA Delft, The Netherlands}

\Email{\href{mailto:w.g.m.groenevelt@tudelft.nl}{w.g.m.groenevelt@tudelft.nl}, \href{mailto:c.c.m.l.wagenaar@tudelft.nl}{c.c.m.l.wagenaar@tudelft.nl}}
\URLaddress{\url{https://fa.ewi.tudelft.nl/~groenevelt/}, \newline
\hspace*{10.5mm}\url{https://fa.ewi.tudelft.nl/~cwagenaar/}}

\ArticleDates{Received June 30, 2022, in final form February 15, 2023; Published online March 05, 2023}

\Abstract{An algebra is introduced which can be considered as a rank 2 extension of the Askey--Wilson algebra. Relations in this algebra are motivated by relations between coproducts of twisted primitive elements in the two-fold tensor product of the quantum algebra $\uql$. It is shown that bivariate $q$-Racah polynomials appear as overlap coefficients of eigenvectors of generators of the algebra. Furthermore, the corresponding $q$-difference operators are calculated using the defining relations of the algebra, showing that it encodes the bispectral properties of the bivariate $q$-Racah polynomials.}


\Classification{20G42; 33D80}

\section{Introduction}	
In this paper, we study a rank 2 version of the Askey--Wilson algebra and its relation to bivariate Askey--Wilson polynomials. The Askey--Wilson algebra $\aw$ was introduced by Zhedanov~\cite{Zhedanov1991} to describe the algebraic structure underlying the Askey--Wilson polynomials \cite{AW1985}. The Askey--Wilson polynomials form an important family of orthogonal polynomials that are on top of the $q$-Askey scheme \cite{KLS} of families of ($q$-)hypergeometric orthogonal polynomials. Every family of polynomials $\{p_n(x)\}$ in this scheme is bispectral; $p_n(x)$ is an eigenfunction of a second order differential or ($q$-)difference operator in $x$, as well as an eigenfunction of a second order difference operator in the degree $n$ (the three-term recurrence relation). It is this bispectrality property of the Askey--Wilson polynomials that is encoded in the Askey--Wilson algebra.\looseness=1

Since its introduction, the Askey--Wilson algebra $\aw$ has appeared in many contexts. For example, $\aw$ is closely related to the rank 1 double affine Hecke algebra of type $\big(C^\vee,C\big)$, Leonard pairs and the $q$-Onsager algebra \cite{BVZ2017, Koornwinder2007, Koornwinder2008, Terwilliger2002, Terwilliger2013,Terwilliger2018}. We refer to \cite{CFGPRV2021} for an overview of the many interpretations of the Askey--Wilson algebra. There is also a close connection between $\aw$ and the centralizer of $\Uq(\slc)$ in $\Uq(\slc)^{\otimes3}$ \cite{CGVZ,GZ93,Huang2017}. In this interpretation, intermediate Casimir operators generate $\aw$, and this also explains the occurrence of $q$-Racah polynomials, which are essentially Askey--Wilson polynomials on a finite set, as Racah coefficients for $\Uq(\slc)$ \cite{Granovskii1993}.

The latter interpretation of $\aw$ in $\Uq(\slc)^{\otimes3}$ has led to the definition of an Askey--Wilson algebra of rank $n$ as the algebra generated by the intermediate Casimir operators in the $(n+2)$-fold tensor product $\Uq(\slc)^{\otimes(n+2)}$ \cite{CL, DeBie2021, DeBie2020, DeClercq19, Post2017}. Multivariate versions \cite{Gasper2005, Gasper2007} of Askey--Wilson polynomials, or more precisely of $q$-Racah polynomials, appear as overlap coefficients in this setting, which is similar to the rank~1 case. This follows from results in \cite{GIV}, where it is shown that the multivariate $q$-Racah polynomials arise as $3nj$-symbols for $(n+2)$-fold tensor product representations of $\Uq(\mathfrak{su}(1,1))$. In~\cite{Iliev2011}, Iliev obtained $q$-difference operators for the multivariate Askey--Wilson polynomials, and these $q$-difference operators lead to a realization of a higher rank Askey--Wilson algebra~\cite{DeBie2021}.

In this paper, we define an algebra $\AW{2}$, a rank 2 version of the Askey--Wilson algebra, as the algebra which is, roughly speaking, generated by $\aw\otimes\aw$ and $\aw$ with certain relations. This definition is motivated by the realization of $\aw$ as a subalgebra of $\Uq(\slc)$ from~\cite{Granovskii1993}: $\aw \subseteq \Uq(\slc)$ is generated by two algebra elements that are essentially Koornwinder's~\cite{Koornwinder1993} twisted primitive elements. In~\cite{Groenevelt2021}, it is shown that Iliev's $q$-difference operators for the multivariate Askey--Wilson polynomials can be obtained from coproducts of twisted primitive elements in discrete series representations of $\Uq(\mathfrak{su}(1,1))$. This suggests that the coproducts of twisted primitive elements generate a rank $n$ version $\AW{n}$ of the Askey--Wilson algebra in $\Uq(\slc)^{\otimes n}$. This construction is different from the construction of the rank~$n$ Askey--Wilson algebra $\mathrm{AW}(n+2)$ in~\cite{DeBie2020}, where a $(n+2)$-fold tensor product is used. An advantage of the construction of $\AW{n}$ is that it allows to use directly representations of $\Uq(\slc)$ to construct representations of a higher rank Askey--Wilson algebra. Furthermore, the explicit relations of the generators suggest how to define the higher rank Askey--Wilson algebra without reference to the larger algebra $\Uq(\slc)^{\otimes n}$.

From a viewpoint of special functions, the related multivariate Askey--Wilson or $q$-Racah polynomials also have different interpretations depending on the construction of the algebra. In connection with $\mathrm{AW}(n+2)$ the $n$-variate $q$-Racah polynomials are overlap coefficients between $(n+1)$-variate orthogonal polynomials (i.e., one more variable), namely $q$-Hahn and $q$-Jacobi polynomials, which arise as nested Clebsch-Gordan coefficients \cite{GIV}. An analogous result for~$\AW{n}$ is that the $n$-variate Askey--Wilson polynomials are overlap coefficients between $n$-variate orthogonal polynomials (i.e., the same number of variables), namely Al-Salam--Chihara polynomials in base $q$ and $q^{-1}$, see \cite{Groenevelt2021}, or, in the setting of $\Uq(\mathfrak{su}(2))$, the $n$-variate $q$-Racah polynomials are overlap coefficients for $n$-variate $q$-Krawtchouk type polynomials.\looseness=1

We consider the rank 2 case in this paper. We show how to obtain bivariate $q$-Racah polynomials and their bispectral properties from a finite-dimensional representation of our algebra $\AW{2}$, so that $\AW{2}$ encodes the bispectral properties of the bivariate $q$-Racah polyno\-mials.\looseness=1

The organization of the paper is as follows. In Section \ref{sec:aw}, we recall well-known results on Zhedanov's Askey--Wilson algebra $\aw$ that will be used in, and also serve as a motivation for, later sections. In particular, we give the definition of~$\aw$, we recall how $q$-Racah polynomials appear as overlap coefficients of eigenvectors of the two generators of $\aw$, and we show that~$\aw$ can be realized as a~subalgebra of the quantized universal enveloping algebra of the Lie algebra $\slc$. Then, in Section~\ref{sec:aw2}, we will be ready to define~$\AW{2}$, a~rank~2 Askey--Wilson algebra. The relations in this algebra come from the relations between (coproducts of) twisted primitive elements and the Casimir in $\uql\otimes \uql$. We also point out the connection with the rank 2 Askey--Wilson algebra introduced in \cite{DeBie2020}. In the last two sections, Sections~\ref{sec:bivarqracah} and~\ref{sec:qdif}, we will construct a representation of~$\AW{2}$ where bivariate $q$-Racah polynomials, similar to the ones defined in~\cite{Gasper2007}, appear as overlap coefficients of the generators of~$\AW{2}$ and find explicitly their difference operator. Calculations concerning the relations between coproducts of twisted primitive elements can be found in Appendices~\ref{app:uquqbivaraw} and~\ref{app:calculationsrelationsAW4}.

\section[Zhedanov's Askey--Wilson algebra AW(3)]{Zhedanov's Askey--Wilson algebra $\boldsymbol{\aw}$}\label{sec:aw}

In this section, we will summarize some results from the Askey--Wilson algebra $\aw$, later in this paper also referred to as the original $\aw$. For more details concerning these results, see~\cite{Zhedanov1991}. Throughout this paper, we assume that $0<q<1$ is fixed.

\subsection{The Askey--Wilson algebra}
Let $\aw$ be the Askey--Wilson algebra defined by Zhedanov \cite{Zhedanov1991}. This is the unital, associative, complex algebra generated by three generators $K_0$, $K_1$ and $K_2$ subject to the relations
\begin{gather}
[K_0,K_1]_q=K_2,\nonumber\\
[K_1,K_2]_q=BK_1 + C_0K_0+D_0,\label{eq:aworiginal}\\
[K_2,K_0]_q=BK_0+C_1K_1 + D_1,\nonumber
	\end{gather}
where $B,C_0,C_1,D_0,D_1\in\RR$ are the structure constants of the algebra and $[\cdot,\cdot]_q$ is the so-called $q$-commutator defined by
\[
[X,Y]_q=q XY - q^{-1}YX.
\]
By substituting the first equation of \eqref{eq:aworiginal} into the second and third, $\aw$ can equivalently be described as the algebra generated by two generators $K_0$ and $K_1$. To ease notation later on, we will replace $K_0$ by $K$ and $K_1$ by $L$ and define the following functions,
\[
\shq(x):=q^x-q^{-x} \qquad \text{and} \qquad \chq(x):=q^x+q^{-x}.
\]
Then $\aw$ is the algebra generated by $K$, $L$ subject to the relations
\begin{gather}
	\chq(2)KLK - K^2L-LK^2= BK + C_1L + D_1,\label{eq:awrelation2}\\
	\chq(2)LKL - L^2K-KL^2= BL + C_0K + D_0 \label{eq:awrelation1}.	 		
\end{gather}
Moreover, $\aw$ has a Casimir element $Q$ given by
\begin{gather}
Q= \big(q^{-1}-q^3\big)KL[K,L]_q + q^2 ([K,L]_q )^2+B(KL+LK)+C_0q^2K^2+C_1q^{-2}L^2\nonumber\\
\hphantom{Q= }{} + D_0\big(1+q^2\big)K+D_1\big(1+q^{-2}\big)L. \label{eq:casimiraw3}
\end{gather}

\subsection[Representations of AW(3)]{Representations of $\boldsymbol{\aw}$} \label{ssec:representationsAW(3)}
On a finite-dimensional vector space, the generators $(K,L)$ will form a Leonard pair, meaning that $L$ acts as a tridiagonal operator on the eigenfunctions of $K$ and, by symmetry of $\aw$, $K$~acts as a tridiagonal operator on the eigenfunctions of~$L$. Moreover, the overlap coefficients of the eigenfunctions of $K$ and $L$ are Askey--Wilson (or $q$-Racah) polynomials where the parameters depend on the structure parameters $B$, $C_0$, $C_1$, $D_0$, $D_1$ and the dimension of the representation. We refer to \cite{TV2004,Vidunas2007} for a description of the relation between Leonard pairs and finite-dimensional representations of~$\aw$. Furthermore, for a complete classification of the finite-dimensional irreducible representations of the universal Askey--Wilson algebra, a central extension of $\aw$, we refer to~\cite{Huang2015}.

Let $V$ be a $(N+1)$-dimensional vector space. The sign of the constants $C_0$ and $C_1$ determine the form of the spectra of the generators $L$ and $K$ respectively. That is, the spectrum of $K$ is of the form $\shq(x)$ if $C_1>0$, $\chq(x)$ if $C_1 <0$ and $q^x$ if $C_1=0$ and similarly for~$L$ and~$C_0$. Each of the 9 cases can be treated similarly. We will focus on the case where $C_0,C_1 <0$. Under the so-called `quantization condition', which we will mention later, there is an irreducible $(N+1)$-dimensional representation of $\aw$. By rescaling our generators $K$ and $L$, we can assign any value, while keeping the same sign, to $C_0$ and $C_1$. We will choose the canonical form of $\aw$, where
\[
C_0=C_1=-(\shq(2))^2.
\]
Then we have the following representation of $\aw$: there exist eigenvectors $\{\psi_n\}_{n=0}^N$ of $K$ such that
\begin{gather}
	K\psi_n=\lambda_n \psi_n,\label{eq:1vd5}\\
	L\psi_n=a_n\psi_{n-1}+b_n\psi_n + a_{n+1}\psi_{n+1}.\label{eq:anbntri}
\end{gather}
The eigenvalues $\lambda_n$ and coefficients $a_n$, $b_n$ are given by
\begin{gather}
	\lambda_n= \shq(2n+p_0+1),\label{eq:3vd5}\\
	a_n^2=\frac{-\prod_{k=0}^3(\shq(2n+p_0)-\shq(p_k))}{\chq(2n+p_0)^2\chq(2n+p_0+1)\chq(2n+p_0-1)},\label{eq:4vd5}\\
	b_n = \frac{B\lambda_n+D_1}{(\lambda_{n}-\lambda_{n-1})(\lambda_{n+1}-\lambda_n)}\label{eq:5vd5}.
\end{gather}
Here, $(p_k)_{k=0}^3$ are roots of the \textit{characteristic polynomial} $\mathscr{P}\colon \CC\to\CC$ of $\aw$ given by
\begin{gather}
 \mathscr{P}(z)= C_0\frac{(\shq(1))^2}{(\chq(1))^4} z^4 +D_0\frac{(\shq(1))^2}{(\chq(1))^2}z^3\nonumber\\
\hphantom{\mathscr{P}(z)=}{}
+\left(\frac{B^2-(\shq(1))^2Q_0}{(\chq(1))^2}+\frac{\left((\shq(1))^2-4\right)C_0C_1}{(\chq(1))^4}\right)z^2 \nonumber\\
\hphantom{\mathscr{P}(z)=}{}
+\left(BD_1-\frac{4C_1D_0}{(\chq(1))^2}\right)z +D_1^2+C_1\frac{B^2+4Q_0}{(\chq(1))^2}-\frac{4C_0C_1^2}{(\chq(1))^4}, \label{eq:awpol}
\end{gather}
where $Q_0$ is the representation value of the Casimir $Q$ from \eqref{eq:casimiraw3}. Moreover, the structure constants $B$, $D_0$ and $D_1$ as well as the value of the Casimir $Q$ can be expressed in terms of the roots of the characteristic polynomial $\mathscr{P}$.
\begin{Proposition}\label{prop:BDQpk}
	We have
	\begin{gather*}
		B = (\shq(1))^2\!\left(\! \shq\!\left(\frac{p_0+p_1}{2}\right)\shq\!\left(\frac{p_2+p_3}{2}\right)- \chq\!\left(\frac{p_0-p_1}{2}\right)\chq\!\left(\frac{p_2-p_3}{2}\right)\right),\\
		D_0= \frac{(\shq(2))^2}{\chq(1)}\!\left(\!\shq\!\left(\frac{p_0+p_1}{2}\right)\chq\!\left(\frac{p_0-p_1}{2}\right)+\shq\!\left(\frac{p_2+p_3}{2}\right)\chq\!\left(\frac{p_2-p_3}{2}\right)\right),\\
		D_1= \frac{(\shq(2))^2}{\chq(1)}\!\left(\!\shq\!\left(\frac{p_2+p_3}{2}\right)\chq\!\left(\frac{p_0-p_1}{2}\right)+\shq\!\left(\frac{p_0+p_1}{2}\right)\chq\!\left(\frac{p_2-p_3}{2}\right)\right),\\
		Q_0=(\shq(2))^2\!\left( \prod_{k=0}^3\shq\!\left(\half p_k\right) + (\shq(1))^2 -\frac{(\shq(1))^2B+D_1^2}{(\shq(1))^2(\shq(2))^2}\right).
	\end{gather*}
\end{Proposition}
\begin{proof}This is similar to \cite[equation~(2.4)]{Zhedanov1991}. We have a slightly different expression since we are in the case $C_0,C_1 <0$ instead of $C_0,C_1 >0$.
\end{proof}

Since our representation is finite-dimensional, we require that $a_0=a_{N+1}=0$. This leads to the `quantization condition' $2(N+1)=p_1-p_0$. When the spectrum of $K$ is fixed, the representation is unique up to equivalence since we can calculate the matrix coefficients of $L$ using \cite[equation~(1.12)]{Zhedanov1991} and \eqref{eq:5vd5}. The eigenvalues $\lambda_n$ satisfy the following recursive relations, which we will need later on,
\begin{gather}
\begin{split}
&\lambda_n^2+\lambda_{n+1}^2 = \chq(2)\lambda_n\lambda_{n+1}-C_1, \\
&\chq(2)\lambda_n=\lambda_{n+1}+\lambda_{n-1},\\
&\lambda_n^2=\lambda_{n+1}\lambda_{n-1} -C_1.
\end{split}\label{eq:spectrumrecursiverelations}
\end{gather}
Moreover, using the symmetry of $\aw$, we can interchange the roles of $K$ and $L$ and get a~similar result. That is, there exist eigenvectors $\{\phi_n\}_{n=0}^N$ of $L$ such that
\begin{equation}\label{eq:1vd5omgedraaid}
		L\phi_m=\mu_m \phi_m,\qquad
		K\phi_m=\tilde{a}_m\phi_{m-1}+\tilde{b}_n\phi_m + \tilde{a}_{m+1}\phi_{m+1},
\end{equation}
where $(\mu_m,\tilde{a}_m,\tilde{b}_m)$ can be found from the formulas for $(\lambda_n,a_n,b_n)$ after interchanging $C_0\leftrightarrow C_1$, $D_0\leftrightarrow D_1$ and $p_k\leftrightarrow s_k$, where $s_k$ are the roots of the polynomial \eqref{eq:awpol} where $C_0\leftrightarrow C_1$ and $D_0\leftrightarrow D_1$. The $p_k$ and $s_k$ are linked via
\begin{gather*}
	s_0=\half \Sigma_p - p_1,\qquad\!\! s_1=\half \Sigma_p - p_0,\qquad\!\! s_2=\half \Sigma_p - p_3,\qquad\!\! s_3=\half \Sigma_p - p_1,\qquad\!\! \Sigma_p = \sum_{k=0}^{3}p_k.
\end{gather*}

\subsection[q-Racah polynomials as overlap coefficients of the generators of AW(3)]{$\boldsymbol{q}$-Racah polynomials as overlap coefficients of the generators of $\boldsymbol{\aw}$} \label{ssec:q-Racah pol}
If the structure parameters $B$, $C_0$, $C_1$, $D_0$ and $D_1$ are real and $a_n^2\geq 0$, we have that $\lambda_n$, $a_n$, $b_n$ are real as well. We then define an inner product on $V$ on the basis $\{\psi_n\}_{n=0}^N$ by
\[
\langle \psi_n,\psi_m\rangle=\delta_{mn}.
\]
Then both $K$ and $L$ are self-adjoint with respect to this inner product. Consequently, the sets of eigenvectors, $\{\psi_n\}_{n=0}^N$ as well as $\{\phi_n\}_{n=0}^N$, will per definition form an orthonormal basis. Define the normalized overlap coefficients $\widehat{P}_n(m)$ to be
\begin{align*}
	\widehat{P}_n(m)=\frac{\langle\phi_m,\psi_n\rangle}{\langle\phi_m,\psi_0\rangle},\qquad n,m=0,\ldots,N.
\end{align*}
Since
\begin{align}
\mu_m\langle\phi_m,\psi_n\rangle & =\langle L\phi_m, \psi_n\rangle =\langle \phi_m, L\psi_n\rangle\nonumber\\
		&=a_n\langle \phi_m,\psi_{n-1}\rangle+b_n\langle \phi_m,\psi_{n}\rangle+a_{n+1}\langle \phi_m,\psi_{n+1}\rangle, \label{eq:origin3term}
\end{align}
the overlap coefficients satisfy
\begin{gather}
	\mu_m \widehat{P}_{n}(m)= a_n \widehat{P}_{n-1}(m) + b_n \widehat{P}_{n}(m) + a_{n+1}\widehat{P}_{n+1}(m).\label{eq:qracahthreeterm}
\end{gather}
Together with the initial condition $\widehat{P}_0(m)=1$ and the convention $\widehat{P}_{-1}(m)=0$, \eqref{eq:qracahthreeterm} generates polynomials $\big(\widehat{P}_n\big)_{n=0}^N$ in the variable $\mu_m$, which can be shown \cite{Zhedanov1991} to be $q$-Racah polynomials with parameters that depend on $(p_k)_{k=0}^3$. The $q$-Racah polynomials $\{R_n\}_{n=0}^N$ are defined \cite{KLS} by a $q$-hypergeometric series,
\begin{gather}
	R_n(y_j;\alpha,\beta,\gamma,\delta;q)=\qhyp{4}{3}{q^{-n},\alpha\beta q^{n+1},q^{-j},\gamma\delta q^{j+1}}{\alpha q,\beta\delta q,\gamma q}{q;q}, \label{eq:defqracah}
\end{gather}
where
\[
y_j=q^{-j}+\gamma\delta q^{j+1}.
\]	
Then we have
\begin{gather}
	\widehat{P}_n(m)= \frac{1}{\sqrt{h_n}} R_n\big(y_m;\alpha,\beta,\gamma,\delta;q^2\big),\label{eq:overlapqracah1}
\end{gather}
where $h_n$ are normalizing constants and
\begin{gather}
	\alpha=-q^{p_0+p_2},\qquad \beta=q^{p_0-p_2},\qquad \gamma=q^{p_0-p_1},\qquad \delta=-q^{p_2+p_3}. \label{eq:qracahpar}
\end{gather}
We will rescale the overlap coefficients $\widehat{P}_n(m)$ such that it becomes a $q$-Racah polynomial without the normalizing constant $\sqrt{1/h_n}$, which will be convenient later on. That is, we define
\begin{gather}
	P_n(m)=\frac{\inprod{\phi_m}{\psi_n}}{\inprod{\phi_m}{\psi_0}} \frac{\inprod{\phi_0}{\psi_0}}{\inprod{\phi_0}{\psi_n}}.\label{eq:overlap}
\end{gather}
This gives on the one hand
\[
P_n(0)=1.
\]
On the other hand, by \eqref{eq:overlapqracah1}, we have
\begin{gather*}
P_n(0)=\frac{1}{\sqrt{h_n}} R_n\big(y_0;\alpha,\beta,\gamma,\delta;q^2\big)\frac{\inprod{\phi_0}{\psi_0}}{\inprod{\phi_0}{\psi_n}}= \frac{1}{\sqrt{h_n}} \frac{\inprod{\phi_0}{\psi_0}}{\inprod{\phi_0}{\psi_n}},
\end{gather*}
since by the definition of the $q$-hypergeometric series we have
\[
R_n\big(y_0;\alpha,\beta,\gamma,\delta;q^2\big)=1.
\]
Therefore,
\[
\frac{\inprod{\phi_0}{\psi_0}}{\inprod{\phi_0}{\psi_n}} = \sqrt{h_n}
\]
and thus
\begin{gather*}
	P_n(m)=R_n\big(y_m;\alpha,\beta,\gamma,\delta;q^2\big). 
\end{gather*}
Since $\{\phi_m\}_{m=0}^N$ is an orthonormal basis, we have
\[
\psi_n=\sum_{m=0}^N\inprod{\psi_n}{\phi_m}\phi_m.
\]
Orthogonality in the degree of the $q$-Racah polynomials comes from $\{\psi_n\}_{n=0}^N$ being orthonormal,
\begin{gather}
	\delta_{n,n'}=\inprod{\psi_{n'}}{\psi_n}=\sum_{m=0}^N \inprod{\psi_{n'}}{\phi_m}\inprod{\phi_m}{\psi_n} = \sum_{m=0}^N \tilde{w}(m,n)P_n(m)\overline{P_{n'}(m)},\label{eq:orthogonalityqracah}
\end{gather}
where $\tilde{w}(m,n)$ is the weight function given by
\begin{gather*}
	\tilde{w}(m,n)=\frac{|\inprod{\phi_m}{\psi_0}\inprod{\phi_0}{\psi_n}|^2}{|\inprod{\phi_0}{\psi_0}|^2}.
\end{gather*}
Since the orthogonality for $q$-Racah polynomials is unique, we have (see, e.g.,~\cite{KLS}),
\begin{gather}
	\tilde{w}(m,n):=\tilde{w}(m,n,p_0,p_1,p_2,p_3;q)=\frac{\rho\big(m,\alpha,\beta,\gamma,\delta;q^2\big)}{h_n(\alpha,\beta,\gamma,\delta;q^2)},\label{eq:weighttilde}
\end{gather}
where
\begin{gather}
	\rho(m,\alpha,\beta,\gamma,\delta;q)= \frac{(\alpha q,\beta\delta q, \gamma q, \gamma\delta q;q)_m\big(1-\gamma\delta q^{2m+1}\big)}{\big(q,\alpha^{-1}\gamma\delta q, \beta^{-1}\gamma q, \delta q;q\big)_m(\alpha\beta q)^m(1-\gamma\delta q)},\label{eq:qracahweight}
\end{gather}
and
\begin{gather}
h_n(\alpha,\beta,\gamma,\delta;q)=\frac{\big(\alpha^{-1},\beta^{-1}\gamma,\alpha^{-1}\delta,\beta^{-1},\gamma\delta q^2;q\big)_\infty}{\big(\alpha^{-1}\beta^{-1} q^{-1},\alpha^{-1}\gamma\delta q, \beta^{-1}\gamma q,\delta q;q\big)_\infty}\nonumber\\
\hphantom{h_n(\alpha,\beta,\gamma,\delta;q)=}{}
\times\frac{(1-\alpha\beta q)(\gamma\delta q)^n \big(q,\alpha\beta\gamma^{-1}q,\alpha\delta^{-1}q,\beta q;q\big)_n}{\big(1-\alpha\beta q^{2n+1}\big)(\alpha q,\alpha\beta q,\beta\delta q,\gamma q;q)_n}. \label{eq:qracahnorm}
\end{gather}
In \eqref{eq:weighttilde}, the roots $(p_k)_{k=0}^3$ are related to $(\alpha,\beta,\gamma,\delta)$ via~\eqref{eq:qracahpar}.

As can be seen, the parameters of the $q$-Racah polynomials can be simply computed from the roots $(p_k)_{k=0}^3$. In Section~\ref{sec:herdefAW}, we will reparametrize the structure constants $B$, $C_0$, $C_1$, $D_0$, $D_1$ such that $(p_k)_{k=0}^3$ can be easily determined from the structure parameters.

\subsection[AW(3) as subalgebra of U\_q(sl(2,C))]{$\boldsymbol{\aw}$ as subalgebra of $\boldsymbol{{\mathcal{U}_{q}(\mathfrak{sl}(2,\mathbb C))}}$}\label{sec:uq}

Let $\Uq=\uql$ be the unital, associative, complex algebra generated by\footnote{Note that the $\hat{K}\in\Uq$ is different from $K\in\aw$.} $\hat{K}$, $\hat{E}$, $\hat{F}$, $\hat{K}^{-1}$ subject to the relations
\begin{gather*}
	\hat{K}\hat{K}^{-1}=1=\hat{K}^{-1}\hat{K},\qquad \hat{K}\hat{E}=q\hat{E}\hat{K},\qquad \hat{K}\hat{F}=q^{-1}\hat{F}\hat{K},\qquad \hat{E}\hat{F}-\hat{F}\hat{E}=\frac{\hat{K}^2-\hat{K}^{-2}}{q-q^{-1}}.
\end{gather*}
The quantum algebra $\Uq$ has a comultiplication $\Delta\colon \Uq \to \Uq \otimes \Uq$ defined on the generators by
\begin{alignat*}{3}
& \Delta\big(\hat{K}\big)=\hat{K}\otimes \hat{K},\qquad && \Delta\big(\hat{E}\big)=\hat{K}\otimes \hat{E} + \hat{E}\otimes \hat{K}^{-1},&\\
& \Delta\big(\hat{K}^{-1}\big)=\hat{K}^{-1}\otimes \hat{K}^{-1},\qquad && \Delta\big(\hat{F}\big)=\hat{K}\otimes \hat{F} + \hat{F} \otimes \hat{K}^{-1}.&
\end{alignat*}
Moreover, $\Uq$ has a Casimir element given by
\begin{gather*}
	\Omega=q^{-1}\hat{K}^2+q\hat{K}^{-2}+(\shq(1))^2\hat{E}\hat{F} = q\hat{K}^2+q^{-1}\hat{K}^{-2}+(\shq(1))^2\hat{F}\hat{E},
\end{gather*}
which commutes with all elements in $\Uq$. Note that $\Delta(\Omega)$ is not a central element of ${\Uq\otimes\Uq}$. Actually, we will see in Section \ref{sec:qdif} that it acts simultaneously as a three-term operator on the eigenvectors of $1\otimes Y_K$ as well as $Y_L\otimes 1$.
Two elements of $\Uq$, which are closely related to Koornwinder's twisted primitive ele\-ments~\cite{Koornwinder1993}, play an important role in this paper. Let $a_{\hat{E}},a_{\hat{F}},a_s,b_{\hat{E}},b_{\hat{F}},b_t\in\CC$, then we define (suggestively) the following elements in $\Uq$,
\begin{gather*}
	Y_{K} =q^\half a_{\hat{E}} \hat{E}\hat{K}+ q^{-\half}a_{\hat{F}} \hat{F}\hat{K} + a_s \hat{K}^2,\\ 
	Y_{L} =q^{-\half} b_{\hat{E}} \hat{E}\hat{K}^{-1}+ q^{\half} b_{\hat{F}} \hat{F}\hat{K}^{-1} + b_t \hat{K}^{-2}.
\end{gather*}	 	
They behave quite well with respect to coproduct,\footnote{The subspaces $\vspan{Y_K,1}$ and $\vspan{Y_L,1}$ satisfy the right and left co-ideal property respectively.}
\begin{align}
	\begin{split}&\Delta(Y_K)=K^2 \otimes Y_K + \big(Y_K -a_s \hat{K}^2\big)\otimes 1,\\
		&\Delta(Y_L)=1\otimes \big(Y_L-b_t \hat{K}^{-2}\big) + Y_L\otimes K^{-2}.\end{split}\label{eq:coproducttpe}
\end{align}
From \eqref{eq:coproducttpe} one can see that $1\otimes Y_K$ commutes with $\Delta(Y_K)$ and $Y_L\otimes 1$ commutes with $\Delta(Y_L)$.

Note that $Y_K$ and $Y_L$ relate to the twisted primitive elements $Y_{s,u}$ and $\tilde{Y}_{t}$ from \cite{Groenevelt2021} by adding a constant such that the unital term cancels\footnote{$Y_{s,u}$ and $\tilde{Y}_{t}$ satisfy more general AW-relations, which are more complicated to work with but are in essence the same object.} and taking $a_{\hat{E}}=u$, $a_{\hat{F}}=-u^{-1}$ and $b_{\hat{E}}=b_{\hat{F}}=1$. Thus, shifting the eigenvalues by a constant while keeping the same eigenvectors.

The elements $Y_K$ and $Y_L$ satisfy the AW-relations \cite{Granovskii1993}, i.e., they satisfy both \eqref{eq:awrelation2} and \eqref{eq:awrelation1}.
\begin{Theorem}\label{thm:tpeaw3}
	The elements $Y_{K},Y_{L}\in\Uq$ satisfy the AW-relations \eqref{eq:awrelation2} and \eqref{eq:awrelation1} with structure parameters
	\begin{gather}
B= (\shq(1))^2 (a_s b_t- \theta\Omega ),\nonumber\\
C_0= - (\chq(1))^2 b_{\hat{E}} b_{\hat{F}},\nonumber\\
C_1= -(\chq(1))^2a_{\hat{E}}a_{\hat{F}},\nonumber\\
D_0 = \chq(1)\big( a_s\Omega b_{\hat{E}} b_{\hat{F}}+(\shq(1))^2b_t\theta\big),\nonumber\\
D_1 = \chq(1)\big( b_t\Omega a_{\hat{E}}a_{\hat{F}}+(\shq(1))^2a_s\theta\big),\label{eq:embeddingconstants}
\end{gather}
where $\theta=-(\shq(1))^{-2}\big(a_{\hat{E}}b_{\hat{F}} + a_{\hat{F}}b_{\hat{E}}\big)$.
\end{Theorem}

\begin{Remark}Note that $B$, $D_0$, $D_1$ are no longer constants due to the appearance of the Casimir~$\Omega$. However, they remain central elements of~$\Uq$. Therefore, $Y_K$ and $Y_L$ formally do not generate an AW(3) algebra. However, they do generate a universal Askey--Wilson algebra \cite{Terwilliger2011}, which is the central extension of $\aw$ where $B$, $C_0$, $C_1$, $D_0$, $D_1$ are central elements of the algebra instead of constants.
\end{Remark}

The representation of almost any instance of $\aw$ can be realized by the pair $(Y_K,Y_L)$. The cases where either $D_0$ or $D_1$ is the only non-zero parameter have to be excluded.
Since in an irreducible representation, $\Omega$ is a constant, we define the algebra $\UqO$, where we add an extra relation to $\Uq$ where $\Omega$ is equal to some $\Omega_0\in\CC$. That is,
\[
\Omega = \Omega_0 \cdot 1, \qquad \Omega_0\in\CC,\qquad 1\in\UqO.
\]
Using this, we can show that almost every instance of $\aw$ is homomorphic to a subalgebra of $\UqO$.
\begin{Proposition}\label{thm:awisouq}
	Let $(B,C_0,C_1,D_0,D_1)\in\RR^5$ be structure constants of $\aw$ such that $D_0$ or $D_1$ is not the only non-zero constant. Then there exist $ Y_K,Y_L\in\UqO$ such that $\aw$ is homomorphic to the subalgebra of $\UqO$ generated by $Y_K$ and $Y_L$.
\end{Proposition}
\begin{proof}
	Let $B,C_0,C_1,D_0,D_1\in\RR$ such that $D_0$ or $D_1$ is not the only non-zero constant. We will show, using Theorem \ref{thm:tpeaw3}, that there exist $Y_K,Y_L\in\UqO$ which satisfy the AW-relation with parameters $B$, $C_0$, $C_1$, $D_0$, $D_1$. Let us fix $\Omega_0\neq0$ and take
\begin{gather}
b_{\hat{E}}=-\frac{C_0}{(\chq(1))^2b_{\hat{F}}},\qquad a_{\hat{F}}=-\frac{C_1}{(\chq(1))^2a_{\hat{E}}}\label{eq:0304pol1}
\end{gather}
	and
	\begin{gather}\label{eq:0304pol2}
		\theta=-(\shq(1))^{-2}\big(a_{\hat{E}}b_{\hat{F}}+a_{\hat{F}}b_{\hat{E}}\big)=-(\shq(1))^{-2}\left( a_{\hat{E}}b_{\hat{F}}+ \frac{C_0C_1}{(\chq(1))^4a_{\hat{E}}b_{\hat{F}}}\right).
	\end{gather}
	For arbitrary $C_0$, $C_1$, $\theta$, we can find $a_{\hat{E}}$, $a_{\hat{F}}$, $b_{\hat{E}}$, $b_{\hat{F}}$ that satisfy \eqref{eq:0304pol1} and \eqref{eq:0304pol2}. Therefore, we have to solve the following three equations
	\begin{gather*}
B=(\shq(1))^2(a_s b_t-\theta\Omega_0),\\
D_0=\chq(1)\left(-\frac{C_0a_s\Omega_0}{(\chq(1))^2} +(\shq(1))^2\theta b_t\right),\\
D_1=-\chq(1)\left(-\frac{C_1b_t\Omega_0}{(\chq(1))^2} +(\shq(1))^2\theta a_s\right),
\end{gather*}
	where we have three variables $\theta$, $a_s$, $b_t$. Let us assume that $C_0,C_1\neq0$. Then we can rewrite above equations to the following form,
	\begin{gather}
x_1-a_1x_2x_3=b_1,\label{eq:x1}\\
x_2-a_2x_1x_3=b_2,\label{eq:x2}\\
x_3-a_3x_1x_2=b_3,\label{eq:x3}
	\end{gather}
	with three variables $x_1=\theta$, $x_2=a_s$, $x_3=b_t$ and where $a_1,a_2,a_3\neq0$, since $C_0,C_1,\Omega_0\neq0$. We will show the above system is consistent. We can eliminate $x_1$ by substituting \eqref{eq:x1} into \eqref{eq:x2}. We obtain
	\begin{gather*}
x_2-a_2(b_1+a_1x_2x_3)x_3=b_2,\qquad
x_3-a_3(b_1+a_1x_2x_3)x_2=b_3,
	\end{gather*}
which we can rewrite to the system
	\begin{gather}
x_2-c_1x_2(x_3)^2=d_1,\label{eq:system21}\\
x_3-c_2(x_2)^2x_3=d_2,\label{eq:system22}
	\end{gather}
	where $c_1,c_2\neq0$, since $a_1$, $a_2$, $a_3$ are non-zero.
	The first equation implies
	\begin{gather}
		x_2=\frac{d_1}{1-c_1(x_3)^2}.\label{eq:x2sol}
	\end{gather}
	If $x_3\neq \pm (c_1)^{-1/2}$, we can substitute this into the second. We then get
	\begin{gather}
		\frac{(x_3-d_2)\big(1-c_1(x_3)^2\big)^2-c_2(d_1)^2x_3}{(1-c_1(x_3)^2)^2}=0.\label{eq:longpol}
	\end{gather}
	Since $c_1\neq 0$, the numerator is a polynomial of degree $5$ in $x_3$. It always has five (in general complex) roots, since $\CC$ is algebraically closed. If it has a root which is not of the form $x_3= \pm (c_1)^{-1/2}$, we can use~\eqref{eq:x2sol} to find $x_2$ and \eqref{eq:x1} to find $x_1$ and we are done. If it has a solution of the form $x_3= \pm (c_1)^{-1/2}$, it follows from~\eqref{eq:longpol} that
	\[
	\pm c_2(d_1)^2(c_1)^{-1/2}=0.
	\]
Since $c_1,c_2\neq0$, this implies $d_1=0$.
However, in that case the system \eqref{eq:system21} and \eqref{eq:system22} can be solved by taking $x_2=0$ and $x_3=d_2$.

	If $C_0=0$, $C_1=0$ or both, the system \eqref{eq:x1}--\eqref{eq:x3} becomes
	\begin{gather*}
		x_1-a_1x_2x_3=b_1,\qquad
		(1-\delta_{C_00})x_2-a_2x_1x_3=b_2,\qquad
		(1-\delta_{C_10})x_3-a_3x_1x_2=b_3.
	\end{gather*}
	This can be solved similarly as before, the only two exceptions being $C_0=C_1=b_1=b_2=0$, $b_3\neq0$ and $C_0=C_1=b_1=b_3=0$, $b_2\neq0$. If we are in one of those cases, say the first, we obtain
	\begin{gather}
		x_1-a_1x_2x_3=0,\label{eq:11}\\
		-a_2x_1x_3=0,\label{eq:22}\\
		-a_3x_1x_2=b_3.\label{eq:33}
	\end{gather}
	The second equation \eqref{eq:22} demands $x_1$ or $x_3$ to be zero, while the third one~\eqref{eq:33} requires $x_1$ and $x_2$ to be non-zero. Therefore, we need $x_3=0$. However, then~\eqref{eq:11} tells us that $x_1=0$ has to hold, which violates \eqref{eq:33}. This corresponds to the case $B=C_0=C_1=D_0=0$ and $D_1\neq 0$, for which no $Y_K$, $Y_L$ exist that generate the corresponding~$\aw$.
\end{proof}

\subsection[Reparametrizing AW(3)]{Reparametrizing $\boldsymbol{\aw}$}\label{sec:herdefAW}
In this subsection we will redefine the structure parameters $(B,C_0,C_1,D_0,D_1)$ of $\aw$ into a~way that is similar to~\eqref{eq:embeddingconstants}. We introduce new structure parameters $(A_0,\dots ,A_5)$ by
\begin{gather}
		B= (\shq(1))^2(A_0A_1 - A_2A_3),\qquad C_0=-(\chq(1))^2A_4,\qquad C_1=-(\chq(1))^2A_5,\nonumber\\
		D_0=\chq(1)(A_0A_3A_4+(\shq(1))^2A_1A_2), \nonumber\\
		D_1=\chq(1) (A_1A_3A_5+(\shq(1))^2A_0A_2).
\label{eq:nieuweA}
\end{gather}
Compared with \eqref{eq:embeddingconstants}, one should think of the correspondence
\[
a_s=A_0,\qquad b_t=A_1,\qquad \theta=A_2,\qquad \Omega=A_3.
\]
This notation is convenient for two reasons. First of all, similarly as in $\aw$ we have that the overlap coefficients of $Y_{K}$ and $Y_{L}$ are univariate Askey--Wilson polynomials~\cite{Groenevelt2021} with parameters depending explicitly on $a_s$, $b_t$, $\Omega$ and $\theta$. Looking at how these parameters appear in~\eqref{eq:embeddingconstants}, we will now see that the overlap coefficients of $K$ and $L$ are $q$-Racah polynomials that depend on $(A_k)_{k=0}^3$ in a simple way.

Secondly, we will see in the \bivariateawalgebra defined later on in Section \ref{sec:aw2}, that $B$, $D_0$ and $D_1$ are not necessarily central elements anymore, but $B$, $D_0$, $D_1$ have precisely the structure of~\eqref{eq:nieuweA}.

To see the convenience of the reparametrized structure parameters, let us consider a finite-dimensional representation from Section~\ref{ssec:representationsAW(3)}, where $N+1$ is the dimension of the representation and $C_0=C_1=-(\shq(2))^2$. Let $(\alpha_k)_{k=0}^3$ be real numbers and let
\begin{gather}
		A_0=\shq(\alpha_0),\qquad A_1=\shq(\alpha_1),\qquad A_2=\chq(\alpha_2),\nonumber\\
A_3=\chq(\alpha_3), \qquad A_4=A_5=(\shq(1))^2.\label{eq:Aalpha}
\end{gather}
Then one can show that $(p_k)_{k=0}^3$ and $(\alpha_k)_{k=0}^3$ relate via a simple linear transformation.
\begin{Proposition}
	Let $C_0=C_1=-(\shq(2))^2$, then we have 	
	\begin{gather}
		p_0=\alpha_0+\alpha_3,\qquad p_1=\alpha_0-\alpha_3,\qquad p_2=\alpha_1+\alpha_2,\qquad p_3=\alpha_1 -\alpha_2,\label{eq:pkak}
	\end{gather}
where $\alpha_3=-N-1$ with $N\in\NN$ the dimension of the representation and $\alpha_0+\alpha_3$ is the starting parameter of the spectrum of $K$.
\end{Proposition}
\begin{proof}
The structure parameters $B$, $D_0$ and $D_1$ can be written in terms of the parameters $(p_k)_{k=0}^3$, see Proposition~\ref{prop:BDQpk}. Together with requiring that the spectrum of $K$ is of the form
	\[
	\shq(2n+p_0),\qquad n=0,\ldots,N,
	\]
	and the quantization condition $p_1-p_0=2(N+1)$ this uniquely determines the roots $(p_k)_{k=0}^3$. We can easily compare the expressions for $B$, $D_0$, $D_1$ from Proposition~\ref{prop:BDQpk} with~\eqref{eq:nieuweA}, using that $\shq(2)=\chq(1)\shq(1)$. This gives
	\begin{gather*}
		p_0+p_1=2\alpha_0,\qquad p_2+p_3=\alpha_1,\qquad p_0-p_1=2\alpha_3,\qquad p_2-p_3=2\alpha_2.
	\end{gather*}
	Rewriting this leads to \eqref{eq:pkak}.
\end{proof}

Furthermore, the condition for finite-dimensional representations given by $p_1-p_0=2(N+1)$, now leads to $\alpha_3=-(N+1)$.
Consequently, substituting this into \eqref{eq:3vd5} gives that the spectrum of $K$ can be written as
\begin{gather}
	\lambda_{n(k)} = \shq(2n(k)+\alpha_0),\qquad n(k)=k-\frac{N}{2},\qquad k\in\{0,\dots,N\},\label{eq:nieuwspectrumK}
\end{gather}
where $k$ will be the degree of the resulting $q$-Racah polynomial. Similarly, the spectrum of $L$ is given by
\begin{gather}
	\mu_{m(j)} = \shq(2m(j)+\alpha_1),\qquad m(j)=j-\frac{N}{2},\qquad j\in\{0,\dots,N\},\label{eq:nieuwspectrumL}
\end{gather}
where $j$ will determine the variable of the $q$-Racah polynomial.
Since the form of the spectra will remain the same in the rank 2 case, we define
\begin{gather}
\lambda_x=\shq(2x+\alpha_0)\qquad\text{and}\qquad \mu_x=\shq(2x+\alpha_1).\label{eq:eigenvalues}
\end{gather}
The parameters $A_4$ and $A_5$ are similar to the roles of $C_0$ and $C_1$. That is, by rescaling the generators $K$ and $L$ we can assign any value, while keeping the same sign, to $A_4$ and $A_5$. This sign determines the form of the spectrum: positive leads to $\shq$, while negative gives $\chq$.

In summary, we have the following interpretation of the new structure parameters:
\begin{itemize}\itemsep=0pt
	\item $\alpha_0\in\RR$ determines the spectrum of $K$,
	\item $\alpha_1\in\RR$ determines the spectrum of $L$,
	\item $\alpha_3\in\ZZ_+$ determines the dimension of the representation,
	\item The sign of $A_4$ determines the form of the spectrum of $L$,
	\item The sign of $A_5$ determines the form of the spectrum of $K$.
\end{itemize}
The overlap coefficients of the eigenfunctions of $K$ and $L$ are $q$-Racah polynomials with parameters that depend in a simple way on $(p_k)_{k=0}^3$ and the $(p_k)_{k=0}^3$ are in turn related to the new structure parameters $(\alpha_k)_{k=0}^3$ via~\eqref{eq:pkak}. Therefore, we see that the $q$-Racah polynomials arising from the overlap coefficients of $K$ and $L$ can be simply computed from $(\alpha_k)_{k=0}^3$.
\begin{Corollary}
	Let $P_{n(k)}(m(j))$ be the overlap coefficients from \eqref{eq:overlap}. If $\alpha_3=-N-1$ with $N\in\NN$, we have
\[
P_{n(k)}(m(j))=R_k\big(y_j;\alpha,\beta,\gamma,\delta;q^2\big),\qquad k,j=0,\dots,N.
\]
	Here, $R_k(y_j;\alpha,\beta,\gamma,\delta;q)$ are the $q$-Racah polynomials from \eqref{eq:defqracah} and
\begin{gather}
\alpha=-q^{\alpha_0+\alpha_1+\alpha_2+\alpha_3},\qquad \beta=q^{\alpha_0-\alpha_1-\alpha_2+\alpha_3},\qquad \gamma=q^{2\alpha_3},\qquad \delta=-q^{2\alpha_1}. \label{eq:qracahparnew}
\end{gather}
\end{Corollary}
\begin{proof}
This is just \eqref{eq:qracahpar} combined with \eqref{eq:pkak}.
\end{proof}

Similar to \eqref{eq:orthogonalityqracah}, we have
\begin{gather*}
	\delta_{n(k),n(k')}=\sum_{j=0}^N w(j,k)P_{n(k)}(m(j))\overline{P_{n(k')}(m(j))},
\end{gather*}
where
\begin{gather}
w(j,k):=w(j,k,\alpha_0,\alpha_1,\alpha_2,\alpha_3;q) =\frac{\rho\big(j,\alpha,\beta,\gamma,\delta;q^2\big)}{h_n\big(\alpha,\beta,\gamma,\delta:q^2\big)}.\label{eq:weightqracahnew}
\end{gather}
Here, the parameters $(\alpha_k)_{k=0}^3$ are related to $(\alpha,\beta,\gamma,\delta)$ via \eqref{eq:qracahparnew} and the functions $\rho$ and $h_n$ are given in \eqref{eq:qracahweight} and \eqref{eq:qracahnorm} respectively.

Moreover, $a_n$ and $b_n$ from \eqref{eq:anbntri} depend explicitly on $\alpha_0$, $\alpha_1$, $\alpha_2$, $N$. This can be seen by substituting \eqref{eq:nieuweA} and \eqref{eq:Aalpha} into formula \eqref{eq:5vd5} for $b_n$ and \eqref{eq:pkak} into formula \eqref{eq:4vd5} for $a_n$. We can simplify the expression for $b_n$ even more by using
\begin{gather}
	\lambda_n-\lambda_{n-1}=\shq(1)\chq(2n+\alpha_0-1).\label{eq:difference eigenvalues}
\end{gather}
Because of the result after substitution, for $\bm{\alpha}=(\alpha_0,\alpha_1,\alpha_2,\alpha_3)\in\RR^4$ we define
\begin{gather}
a_n^2(\bm{\alpha})=\frac{-\prod_{k=0}^3(\shq(2n+\alpha_0-1)-\shq(p_k))}{\chq(2n+\alpha_0-1)^2\chq(2n+\alpha_0)\chq(2n+\alpha_0-2)},\label{eq:annew}\\
b_n(\bm{\alpha})=\frac{ (A_1 -A_2A_3)\lambda_n+\chq(1) (A_1A_3+A_0A_2)}{\chq(2n+\alpha_0-1)\chq(2n+\alpha_0+1)},\label{eq:bnnew}
\end{gather}
where $(p_k)_{k=0}^3$ and $(A_k)_{k=0}^3$ depend on $(\alpha_k)_{k=0}^3$ via \eqref{eq:pkak} and \eqref{eq:Aalpha} respectively. When $C_0>0$ and $C_1<0$, the formulas for $a_n^2$ and $b_n$ change slightly. For $a_n^2$, we have to remove the minus sign in front of the product in the numerator and interchange $\alpha_1\leftrightarrow \alpha_2$. The formula for $b_n$ changes in the sense that $(A_k)_{k=0}^3$ depend on $(\alpha_k)_{k=0}^3$ in a different way. Therefore we define
\begin{gather}
	\tilde{a}_n^2(\tilde{\alpha}_0,\tilde{\alpha}_1,\tilde{\alpha}_2,\tilde{\alpha}_3) =-a_n^2(\tilde{\alpha}_0,\tilde{\alpha_2},\tilde{\alpha_1},\tilde{\alpha}_3),\label{eq:annewtilde}\\
	\tilde{b}_n(\bm{\alpha})=\frac{ \big(\tilde{A}_1 -\tilde{A}_2\tilde{A}_3\big)\lambda_n+\chq(1) \big(\tilde{A}_1\tilde{A}_3+\tilde{A}_0\tilde{A}_2\big)}{\chq(2n+\alpha_0-1)\chq(2n+\alpha_0+1)},\label{eq:bnnewtilde}
\end{gather}
where
\begin{gather*}
	\tilde{A}_0=\shq(\alpha_0),\qquad \tilde{A}_1=\chq(\alpha_1),\qquad \tilde{A}_2=\shq(\alpha_2),\qquad \tilde{A}_3=\chq(\alpha_3).\label{eq:Aalphatilde}
\end{gather*}

To end this section let us make the connection with the Askey--Wilson algebra as defined in~\cite{CFGPRV2021}.
Let us write
\[
\Lambda_1=A_1, \qquad \Lambda_2=A_3, \qquad \Lambda_3=A_0,\qquad
\Lambda_{12} = L, \qquad \Lambda_{23}=K, \qquad \Lambda_{123}=-A_2,
\]
and set $A_4=A_5=-(\shq(1))^2$. Define
\[
\Lambda_{13}= -\frac{[\Lambda_{12},\Lambda_{23}]_q}{\shq(2)} + \frac{\Lambda_1 \Lambda_3+ \Lambda_2 \Lambda_{123}}{\chq(1)},
\]
then the AW-relations \eqref{eq:awrelation2} and \eqref{eq:awrelation1} can be written as
\begin{gather*}
	\Lambda_{12}= -\frac{[\Lambda_{23},\Lambda_{13}]_q}{\shq(2)} + \frac{\Lambda_2 \Lambda_3+ \Lambda_1 \Lambda_{123}}{\chq(1)}, \qquad
	\Lambda_{23}= -\frac{[\Lambda_{13},\Lambda_{12}]_q}{\shq(2)} + \frac{\Lambda_1 \Lambda_2+ \Lambda_3 \Lambda_{123}}{\chq(1)}.
\end{gather*}
Moreover, for the Askey--Wilson algebra inside $\Uq$ generated by $Y_L$ and $Y_k$ it follows from \cite[Theorem~2.17]{Terwilliger2011a} that the corresponding Casimir element $Q$ \eqref{eq:casimiraw3} of the Askey--Wilson algebra can be simplified to
\[
Q= (\chq(1))^2-\Lambda_{123}^2 - \Lambda_1^2-\Lambda_2^2 - \Lambda_3^2 - \Lambda_{123}\Lambda_1\Lambda_2\Lambda_3.
\]
This implies that $Y_L$ and $Y_K$ generate a \textit{special} Askey--Wilson algebra, as defined in \cite[Section~2.1]{CFGPRV2021}, inside~$\Uq$.

\section[Construction of the rank 2 Askey-Wilson algebra AW\_2]{Construction of the \bivariateawalgebra $\boldsymbol{\AW{2}}$}\label{sec:aw2}

\subsection[Coproducts of Y\_K and Y\_L and the Askey-Wilson algebra]{Coproducts of $\boldsymbol{Y_K}$ and $\boldsymbol{Y_L}$ and the Askey--Wilson algebra}\label{subsec:coproduc}

In \cite{Groenevelt2021}, it was shown that the overlap coefficients of coproducts of $Y_{K}$ and $Y_{L}$ are multivariate Askey--Wilson polynomials. Moreover, $q$-difference operators for which these multivariate Askey--Wilson polynomials are eigenfunctions can be realized by these coproducts of twisted primitive elements. One can check that these coproducts also satisfy the AW-relations (Theorem~\ref{thm:tabel}). This then will be our motivation for the defining relations of the \bivariateawalgebra $\AW{2}$.

Let us consider the following elements in $\Uq\otimes\Uq$,
\begin{gather*}
	1\otimes Y_{K},\quad \Delta(Y_{K}),\quad Y_{L}\otimes 1,\quad \Delta(Y_{L})\quad \text{and}\quad \Delta(\Omega).
\end{gather*}
We want to know how these elements relate to each other. Two easy observations are that $1\otimes Y_{K}$ and $Y_{L}\otimes 1$ commute and that $\Delta(Y_{K})$ and $\Delta(Y_{L})$ satisfy the AW-relations, since $\Delta$ is an algebra homomorphism. That is,
\begin{gather*}
\chq(2)\Delta(Y_{L})\Delta(Y_{K})\Delta(Y_{L}) - \Delta(Y_{L})^2\Delta(Y_{K})-\Delta(Y_{K})\Delta(Y_{L})^2 \\
\qquad{}=\Delta(B)\Delta(Y_{L}) + C_0\Delta(Y_{K}) + \Delta(D_0), \\
\chq(2)\Delta(Y_{K})\Delta(Y_{L})\Delta(Y_{K}) - \Delta(Y_{K})^2\Delta(Y_{L})-\Delta(Y_{L})\Delta(Y_{K})^2\\
\qquad{}=\Delta(B)\Delta(Y_{K}) + C_1\Delta(Y_{L}) + \Delta(D_1),
\end{gather*}
where $B$, $C_0$, $C_1$, $D_0$, $D_1$ can be found in \eqref{eq:embeddingconstants}. Notice that $C_0$ and $C_1$ are constants, thus
\[
\Delta(C_0)=C_0 (1\otimes 1)\ \text{ and }\ \Delta(C_1)=C_1 (1\otimes 1).
\]
However, $\Delta(B)$, $\Delta(D_0)$ and $\Delta(D_1)$ are not central elements of $\Uq\otimes \Uq$ anymore due to the appearance of $\Delta(\Omega)$. They do not even commute with our four generators,
since $\Delta(\Omega)$ does not commute with $1\otimes Y_{K}$ or $Y_{L}\otimes 1$. We do see that $\Delta(\Omega)$ `locally' commutes, i.e., it commutes with~$\dyzero$ and~$\dyone$, the elements of the AW-relations it appears in as a structure parameter. Also, we already noted that~$\yzero$ and $\dyzero$ commute and $\yone$ and $\dyone$ as well because of~\eqref{eq:coproducttpe}.

Let us now study how $\yone$ and $\dyzero$ relate. It turns out that $(\yone,\dyzero)$ satisfy the $\aw$ relations~\eqref{eq:awrelation2} and~\eqref{eq:awrelation1} where the structure `parameters' contain the `locally' commuting operators $\dyone$ and $\yzero$. A similar result is (non surprisingly) true for $(\yzero,\dyone)$. Also, both couples $(\Delta(\Omega),\yzero)$ and $(\Delta(\Omega),\yone)$ satisfy the AW-relations. This is summarized in the next theorem.
\begin{Theorem}\label{thm:tabel}
	Each pair of the elements $\yzero,\dyzero,\yone,\dyone,\Delta(\Omega)\in\Uq\otimes\Uq$ either commutes or satisfies the AW-relations \eqref{eq:awrelation2} and \eqref{eq:awrelation1}. The pairs of non-commuting elements satisfy these relations with the structure parameters in Table~{\rm \ref{tab:tpegen}}.
	\begin{table}[h]\centering\renewcommand{\arraystretch}{1.2}\small
		\begin{tabular}{|l|l|l|l|l|l|l|l|}\hline
			\text{Generator 1} & \text{Generator 2} & $A_0$ & $A_1$ & $A_2$ & $A_3$ & $A_4$ & $A_5$ \\ \hline
			$\dyzero$ & $\dyone$ & $a_s$ & $b_t$ & $\theta$ & $\Delta\left(\Omega\right)$ & $b_{\hat{E}}b_{\hat{F}}$ & $a_{\hat{E}}a_{\hat{F}}$ \\ \hline
			$\yzero$ & $\dyone$ & $a_s$ & $\yone$ & $\theta$ & $1\otimes\Omega$ & $b_{\hat{E}}b_{\hat{F}}$ & $a_{\hat{E}}a_{\hat{F}}$ \\ \hline
			$\dyzero$ & $\yone$ & $\yzero$ & $b_t$ & $\theta$ & $\Omega\otimes 1$ & $b_{\hat{E}}b_{\hat{F}}$ & $a_{\hat{E}}a_{\hat{F}}$ \\ \hline
			$\yzero$ & $\Delta\left(\Omega\right)$ & $a_s$ & $\Omega\otimes 1$ & $-\dyzero$ & $1\otimes\Omega$ & $-\shq(1)^{2}$ &$a_{\hat{E}}a_{\hat{F}}$ \\ \hline
			$\Delta\left(\Omega\right)$ & $\yone$ & $1\otimes\Omega$ & $b_t$ & $-\dyone$ & $\Omega\otimes 1$ & $ b_{\hat{E}}b_{\hat{F}}$ &$-\shq(1)^{2}$ \\ \hline
		\end{tabular}
		\caption{Askey--Wilson algebra relations of twisted primitive elements.}\label{tab:tpegen}
	\end{table}
\end{Theorem}

\begin{proof}This is a direct computation in $\Uq\otimes\Uq$. For readability, these calculations can be found in Appendix~\ref{app:uquqbivaraw}.
\end{proof}

\begin{Remark}
	Let us consider the tensor product of two irreducible representations of $\mathcal U_q$. Clebsch-Gordan coefficients for a basis of eigenvectors of $\Delta(Y_K)$ which are also eigenvectors of $1 \otimes Y_K$, can be obtained from diagonalizing $\Delta(\Omega)$ acting on such eigenvectors. From Table~\ref{tab:tpegen}, we see that $\Delta(\Omega)$ and $1 \otimes Y_K$ generate an Askey--Wilson algebra, and note that the structure `parameters' $A_1=\Omega\otimes 1$, $A_2=\Delta(Y_K)$ and $A_3=1\otimes \Omega$ act as multiplication operators in this case. As a consequence, see Section~\ref{ssec:q-Racah pol}, the Clebsch-Gordan coefficients are essentially $q$-Racah polynomials. Using this reasoning in the case of tensor products of (infinite-dimensional) irreducible representations of $\mathcal U_q(\mathfrak{su}(1,1))$, one expects the Clebsch-Gordan coefficients for simultaneous eigenvectors of $\Delta(Y_K)$ and $1\otimes Y_K$ to be eigenfunctions of an Askey--Wilson $q$-difference operator. Indeed, in the case of a tensor product of discrete series representations of $\mathcal U_q(\mathfrak{su}(1,1))$ the Clebsch-Gordan are essentially Askey--Wilson polynomials, see \cite{Groenevelt2004, KvdJ}.
\end{Remark}

\subsection[Defining the rank 2 Askey-Wilson algebra AW\_2]{Defining the \bivariateawalgebra $\boldsymbol{\AW{2}}$}

Inspired by Theorem \ref{thm:tabel}, we will define the \bivariateawalgebra $\AW{2}$. Constructing this algebra will be done without any dependence on $\Uq$. We will introduce elements in $\AW{2}$ which resemble $Y_K$, $Y_L$, $\Omega$ in $\Uq$ as well as their coproducts. Defining a comultiplication is not needed in $\AW{2}$. This construction of $\AW{2}$ will proceed in three steps. First, we will set up two original $\aw$ algebras and define the algebra $\mathscr{A}$ to be their tensor product. Then, we let $\awc{1}{M_2}$ be the Askey--Wilson algebra that resembles the first row of Table~\ref{tab:tpegen}. Lastly, we will combine these algebras and define $\AW{2}$ as the algebra generated\footnote{This is the free product algebra of the algebras $\mathscr{A}$ and $\mathscr{B}$ subject to the extra relations.} by the two algebras~$\mathscr{A}$ and~$\awc{1}{M_2}$ subject to the Askey--Wilson algebra relations that appear in rows 2--5 of Table~\ref{tab:tpegen}.

To have a short notation for when two elements of an algebra satisfy the AW-relations, we will introduce the function $\awr$.
\begin{Definition}
	Let $A$ be an algebra over $\CC$ and $A^{ n}$ be the $n$-ary Cartesian product of $A$. Define the function $\awr\colon A^{8}\to A^{2}$ by
	\begin{gather*}
		\awr(K,L\mid A_0,A_1,A_2,A_3,A_4,A_5)=\left(\begin{matrix}\chq(2) KLK - K^2L -LK^2 \\ \chq(2)LKL - L^2K -KL^2 \end{matrix}\right) \\
		-\left(\begin{matrix}	(\shq(1))^2(A_0A_1\! -\!A_2A_3)K \!-\! (\chq(1))^2A_5L \!+\! \chq(1)(A_1A_3A_5\!+\!(\shq(1))^2A_0A_2)		\\ (\shq(1))^2(A_0A_1\! -\!A_2A_3)L \!-\!(\chq(1))^2A_4K \!+\!\chq(1)(A_0A_3A_4\!+\!(\shq(1))^2A_1A_2) \end{matrix}\right)
	\end{gather*}
	for $K,L,A_0,A_1,A_2,A_3,A_4,A_5\in A$. We then call $K$, $L$ the generators and $(A_k)_{k=0}^5$ the structure parameters. Moreover, we say that the structure parameters are locally central if $(A_k)_{k=0}^5$ commute with each other as well as with the generators $K$ and $L$ of that relation. That is,
	\begin{gather*}
		[A_k,A_l]=[A_k,K]=[A_k,L]=0 \qquad \text{for all} \quad k,l\in\{0,1,2,3,4,5\},
	\end{gather*}
	where $[\cdot,\cdot]$ is the regular commutation bracket defined by
	\[
	[X,Y]=XY-YX.
	\]
\end{Definition}
\begin{Remark}
	We have that $\awr(K,L|A_0,A_1,A_2,A_3,A_4,A_5)=0$ if and only if two elements $K$,~$L$ of an algebra satisfy the AW-relations with parameters $A_0,\dots, A_5$.
\end{Remark}
From now on, fix constants $A_0,A_1,A_2,A_{N_1},A_{N_2},\sigma_K,\sigma_L\in\RR$. Later on, $A_{N_i}$ will determine the dimension of the vector space $V_i$ by taking
\[
A_{N_i}=\chq(-N_i-1).
\]
Let us first introduce the two original $\aw$ algebras. For $i\in\{1,2\}$, let $\awc{1}{A_{N_i}}$ be the algebra generated by $K_1$ and $L_1$ subject to the relations
\[
\awr(K_1,L_1\mid A_0,A_1,A_2,A_{N_i},\sigma_L,\sigma_K)=0.
\]
Then we let
\[
\mathscr{A}=\awc{1}{A_{N_1}}\otimes\awc{1}{A_{N_2}}.
\]
Secondly, we need the Askey--Wilson algebra $\awc{1}{M_2}$. This is the algebra generated by~$K_2$,~$L_2$ and $M_2$ subject to the relations
\begin{gather}
	\awr(K_2,L_2\mid A_0,A_1,A_2,M_2,\sigma_L,\sigma_K)=0,\label{eq:aw1(M2)}
\end{gather}
where $M_2$ is a central element of $\awc{1}{M_2}$. In the setting of Section~\ref{subsec:coproduc}, $M_2$ would be the coproduct of the Casimir of $\Uq$, which is central with respect to algebra $\Delta(\Uq)$. Now we are ready to define $\AW{2}$, where the relations are motivated by the commutation relations of $\Uq\otimes\Uq$ in rows 2--5 of Table~\ref{tab:tpegen}.
\begin{Definition} \label{Def:AW2}
	Let $\AW{2}$ be the unital, associative, complex algebra generated by the algebras~$\mathscr{A}$ and~$\awc{1}{M_2}$ subject to the relations
	\begin{gather}
		\awr(1\otimes K_1,L_2\mid A_0,L_1\otimes 1,A_2,A_{N_2},\sigma_L,\sigma_K)=0,\label{eq:aw4}\\
		\awr(K_2,L_1\otimes 1\mid 1\otimes K_1,A_1,A_2,A_{N_1},\sigma_L,\sigma_K)=0,\label{eq:aw5}\\
		\awr\big(1\otimes K_1,M_2\mid A_0,A_{N_1},K_2,A_{N_2},-\shq(1)^{2},\sigma_K\big)=0,\label{eq:aw6}\\
		\awr\big(M_2,L_1\otimes 1\mid A_{N_2},A_1,L_2,A_{N_1},\sigma_L,-\shq(1)^{2}\big)=0,\label{eq:aw7}
	\end{gather}
	where all structure parameters are locally central.
\end{Definition}
\begin{Remark}
	Since $1\otimes K_1$ is locally central in \eqref{eq:aw5}, it commutes with $K_2$. Similarly, $L_1\otimes 1$ commutes with $L_2$. Also, $M_2$ commutes with both $K_2$ and $L_2$.
\end{Remark}
\noindent The correspondence between elements in $\Uq\otimes\Uq$ and $\AW{2}$ can be found in Table~\ref{tab:correspondenceuqaw}.
\begin{table}[h]\centering \renewcommand{\arraystretch}{1.2}\small
	\begin{tabular}{|@{\,\,}l@{\,\,}|l|l|l|l|l|l|l|l|l|l|l|l@{\,\,}|}
		\hline
		Algebra & \multicolumn{5}{c|}{Generators} & \multicolumn{7}{c|}{Central elements} \\ \hline
		$ \Uq \otimes \Uq $ &$ 1 \otimes Y_K $& $ \Delta(Y_K) $& $ Y_L\otimes $ 1& $ \Delta(Y_L) $& $ \Delta(\Omega) $& $ a_s $& $ b_t $& $ \theta $& $ \Omega \otimes 1$ & $ 1 \otimes \Omega $& $ b_Eb_F $& $ a_Ea_F $ \\ \hline
		$ \AW{2} $ & $ 1\otimes K_1 $ &$ K_2 $&$ L_1 \otimes 1 $&$ L_2 $& $ M_2 $& $ A_0 $& $ A_1 $& $ A_2 $& $ A_{N_1} $&$ A_{N_2} $&$ \sigma_L $&$ \sigma_K $ \\ \hline
	\end{tabular}
\caption{Correspondence between $\Uq\otimes\Uq$ and the generators of $\AW{2}$. }\label{tab:correspondenceuqaw}
\end{table}

\begin{Remark}\label{rem:correspondenceaw4aw2}
	There exists a~homomorphism from $\AW{2}$ to the rank $2$ Askey--Wilson algebra $\mathrm{AW}(4)$ defined in \cite{DeBie2020}. Indeed, let $\Lambda_A$, $A \subseteq \{1,2,3,4\}$ be defined as in \cite[Section~2.3]{DeBie2020}. It can be checked that the $\Lambda$'s satisfy the $\AW{2}$-relations if we set
	\begin{gather}
		\Lambda_{\{1\}}=A_1, \qquad \Lambda_{\{2\}}=A_{N_1},\qquad \Lambda_{\{3\}}=A_{N_2}, \qquad \Lambda_{\{4\}}=A_0, \nonumber\\
		\Lambda_{\{1,2\}}=L_1\otimes 1, \qquad \Lambda_{\{2,3\}} =M_2, \qquad \Lambda_{\{3,4\}} =1\otimes K_1, \label{eq:correspondenceaw2aw4}\\
		\Lambda_{\{1,2,3\}}=L_2, \qquad \Lambda_{\{1,2,3,4\}} =A_2, \qquad \Lambda_{\{2,3,4\}}=K_2.\nonumber
	\end{gather}
	However, this correspondence does not give an isomorphism since $\mathrm{AW}(4)$ has more relations. Namely, the relations between $\Lambda_A$, $\Lambda_B$ and $\Lambda_C$ where $(A,B,C)$ is any cyclic permutation of
	\begin{gather}
		(\{1,3\},\{3,4\},\{1,4\})\qquad \text{or}\qquad (\{1,2\},\{2,4\},\{1,4\})\label{eq:aw4relationswithholes},
	\end{gather}
	i.e., the relations where two $\Lambda$'s have `holes' in it. An interesting question is whether above relations hold in the subalgebra of $\Uq\otimes \Uq$ generated by twisted primitive elements, the Casimir~$\Omega$ and their coproducts. That is, we can use Table~\ref{tab:correspondenceuqaw} and~\eqref{eq:correspondenceaw2aw4} to obtain a~correspondence between that subalgebra of $\Uq\otimes\Uq$ and~$\mathrm{AW}(4)$. It can be show that the extra relations \eqref{eq:aw4relationswithholes} indeed hold there, which requires large computations in $\Uq\otimes \Uq$. For one of the relations we have written out parts of the computation in Appendix~\ref{app:calculationsrelationsAW4}.
	
	The relations in \cite{DeBie2020} suggest how the relations of a rank $N$ Askey--Wilson algebra inside $\Uq^{\otimes N}$ may look. This would simplify checking relations in $\Uq\tensor{N}$, which is done in Appendix \ref{app:uquqbivaraw} for the `easy' case $N=2$. Let us define $\Delta^n\colon \Uq\tensor{n}\to\Uq\tensor{n+1}$ recursively by
	\begin{gather*}
		\Delta^n=\big(1\tensor{n-1}\otimes\Delta\big)\Delta^{n-1},
	\end{gather*}
	with $\Delta^0$ the identity on $\Uq$. Note that by the coassociativity of $\Delta$, this is equal to
	\begin{gather*}
		\Delta^n=\big(\Delta\otimes 1\tensor{n-1}\big)\Delta^{n-1}.
	\end{gather*}
	Denote by $[i:j]$ the set of consecutive integers $\{i,i+1,\dots,j-1,j\}$. Then we conjecture the following correspondence between coproducts of twisted primitive elements and Casimirs in $\Uq\tensor{N}$ and $\Lambda_A$'s, $A\subseteq \{1,\dots,N+2\}$, in $\mathrm{AW}(N+2)$,
	\begin{gather*}
		\Lambda_{\{1\}}=a_\sigma,\qquad \Lambda_{\{N+2\}}=b_t,\qquad \Lambda_{[1:N+2]}=\theta,\\
		\Lambda_{[1:n+1]}=\Delta^{n-1}(Y_L)\otimes 1\tensor{N-n},\qquad n=0,\dots,N,\\
		\Lambda_{[i:j]}=1\tensor{i-2}\otimes \Delta^{j-i}(\Omega)\otimes 1\tensor{N+1-j},\qquad 2\leq i \leq j\leq N+1,\\
		\Lambda_{[n+1:N+2]}=1\tensor{n-1}\otimes \Delta^{N-n}(Y_K),\qquad n=1,\dots,N.
	\end{gather*}
\end{Remark}

\subsection[Finite-dimensional representations of AW\_2]{Finite-dimensional representations of $\boldsymbol{\AW{2}}$}

By Theorem~\ref{thm:tabel}, the elements $1\otimes Y_K$, $Y_L\otimes 1$, $\Delta(Y_K)$, $\Delta(Y_L)$ and $\Delta(\Omega)$ in $\Uq\tensor{2}$ satisfy all relations of $\AW{2}$. Therefore, we automatically get existence of a representation of $\AW{2}$ from the representation of $\Uq\tensor{2}$ if we take $A_{N_1}$ and $A_{N_2}$ to be the constants of the representation value of $\Omega\otimes1$ and $1\otimes \Omega$ in $\Uq\tensor{2}$.
\begin{Corollary}\label{cor:aw2reprexist}
	Let $V_1$ and $V_2$ be vector spaces with $\dim(V_k)=N_k+1$ for $k=1,2$. If $A_{N_1}=\chq(N_1+1)$ and $A_{N_2}=\chq(N_2+1)$, there exists a representation of $\AW{2}$ on $V_1\otimes V_2$.
\end{Corollary}
Moreover, we can show that there exists a representation of $\AW{2}$, where $1\otimes K_1$, $K_2$, $L_1\otimes 1$,~$L_2$ and~$M_2$ are self-adjoint in the case $\sigma_K,\sigma_L\geq0$ using the representation theory of $\Uq\tensor{2}$ and the $*$-structures of the Hopf algebra $\Uq$.
\begin{Proposition}\label{prop:*reprexists}
	Let $V_1$, $V_2$, $A_{N_1}$ and $A_{N_2}$ as in Corollary~{\rm \ref{cor:aw2reprexist}}, $A_0,A_1,A_2\in\RR$ and let ${\sigma_K,\sigma_L\geq 0}$. Then there exist a representation of $\AW{2}$ on $V_1\otimes V_2$ and an inner product on $V_1\otimes V_2$ such that $1\otimes K_1$, $K_2$, $L_1\otimes 1$, $L_2$ and $M_2$ are self-adjoint.			
\end{Proposition}
\begin{proof}
	Let $*$ be any $*$-structure on $\Uq$, then we have
	\begin{gather*}
		Y_K^*=q^{-\half} \overline{a}_{\hat{E}} \hat{K}^*\hat{E}^*+ q^{\half}\overline{a}_{\hat{F}} \hat{K}^*\hat{F}^* + \overline{a}_s \big(\hat{K}^*\big)^2,\\
		Y_{L}^* =q^{\half} \overline{b}_{\hat{E}} \big(\hat{K}^{-1}\big)^*\hat{E}^*+ q^{-\half} \overline{b}_{\hat{F}} \big(\hat{K}^{-1}\big)^*\hat{F}^* + \overline{b}_t \big(\big(\hat{K}^{-1}\big)^*\big)^2.
	\end{gather*}
	There are 2 inequivalent $*$-structures on $\Uq$ in the case $0<q<1$,
	\begin{enumerate}\itemsep=0pt
		\item[(i)] $\uq$ defined by $\hat{K}^*=\hat{K},$ $\hat{E}^*=\hat{F}$, $\hat{F}^*=\hat{E}$,
		\item[(ii)] $\uqeen$ defined by $\hat{K}^*=\hat{K},$ $\hat{E}^*=-\hat{F}$, $\hat{F}^*=-\hat{E}$.
	\end{enumerate}
	Since $\uqeen$ has no irreducible finite-dimensional representations except the trivial one, we will focus on $\uq$. The Casimir $\Omega$ is self-adjoint in $\uq$, which implies $\Delta(\Omega)$ is as well. In order for $Y_K$ and $Y_L$ to be self-adjoint in $\uq$, we require 	$\overline{a}_{\hat{E}} = a_{\hat{F}}$, $\overline{b}_{\hat{E}}=b_{\hat{F}}$ and $a_s,b_t\in\RR$. Then we have, using Table~\ref{tab:correspondenceuqaw},
	\[
	\sigma_K=|a_{\hat{E}}|^2 \geq 0\qquad \text{and}\qquad \sigma_L=|b_{\hat{E}}|^2 \geq 0
	\]
	and $A_0,A_1,A_2\in\RR$. For each dimension $N+1\in\NN$, the algebra $\uq$ has an irreducible $*$-representation, see, e.g.,~\cite{Koelink1996}. This in turn defines a $*$-representation on $\uq\otimes \uq$. By Theorem~\ref{thm:tabel}, this is automatically a representation of $\AW{2}$ in which $1\otimes K_1$, $K_2$, $L_1\otimes 1$, $L_2$ and $M_2$ are self-adjoint.	
\end{proof}

\begin{Remark}In the following sections we use the above defined finite-dimensional representation of $\AW{2}$. For existence of this representation we referred to $*$-representations of $\Uq$. However, let us stress that the calculations in the following sections only make use of the $\AW{2}$-relations, and other than for existence of the representations, no representation theory of $\Uq$ is needed.
\end{Remark}

\section[Bivariate Askey--Wilson polynomials as overlap coefficients of AW\_2]{Bivariate Askey--Wilson polynomials\\ as overlap coefficients of $\boldsymbol{\AW{2}}$}\label{sec:bivarqracah}

We will now show that the overlap coefficients of eigenvector of $K_2$ and $L_2$ are bivariate $q$-Racah polynomials similar to the ones defined by Gasper and Rahman \cite{Gasper2007}. We will construct a~representation of~$\AW{2}$ on a tensor product of vector spaces $V_1$ and $V_2$, which have dimensions $N_1+1$ and $N_2+1$ respectively.

A representation of $\AW{2}$ on $V_1\otimes V_2$ exists by Corollary \ref{cor:aw2reprexist}. Note that $\awc{1}{A_{N_1}}\otimes\awc{1}{A_{N_2}}$ has a natural representation on $V_1\otimes V_2$ coming from the representations from $\awc{1}{A_{N_1}}$ on $V_1$ and $\awc{1}{A_{N_2}}$ on $V_2$ given by \eqref{eq:1vd5}--\eqref{eq:5vd5}. We will extend this representation of $\awc{1}{A_{N_1}}\otimes\awc{1}{A_{N_2}}$ to $\AW{2}$.
Let us fix $\sigma_L=\sigma_K=(\shq(1))^2$ and $A_{N_i}=\chq(-N_i-1)$. Then we are in the same setting\footnote{That is, $C_0=C_1=-(\shq(2))^2$.} as before. Then, by \eqref{eq:1vd5} and \eqref{eq:nieuwspectrumK} we get a basis for $V_2$ given by $\{\psi_{n_2(k)}\}_{k=0}^{N_2}$, where
\[
n_2(k)=k-\half N_2 .
\]
These are the eigenvectors of $K_1\in\awc{1}{A_{N_2}}$ corresponding to eigenvalue $\lambda_{n_2(k)}$. We will often write
\[
\psi_{n_2}:=\psi_{n_2(k)} \qquad \text{and}\qquad \lambda_{n_2}:=\lambda_{n_2(k)}
\]
when the index $k$ is not relevant.

Similarly, by \eqref{eq:1vd5omgedraaid} and \eqref{eq:nieuwspectrumL} there is a basis $\{\phi_{m_1(j)}\}_{j=0}^{N_1}$ for $V_1$, which are the eigenvectors of $L_1\in\awc{1}{A_{N_1}}$ with eigenvalue $\mu_{m_1(j)}$. Here
\[
m_1(j)=j-\half N_1.
\]
Similar to before, we will write
\[
\phi_{m_1}:=\phi_{m_1(j)} \qquad \text{and}\qquad \mu_{m_1}:=\mu_{m_1(j)}
\]
when the index $j$ is not relevant.

Thus, if $v_1\in V_1$ and $v_2\in V_2$, we have
\begin{gather*}
(1\otimes K_1)(v_1\otimes \psi_{n_2})=\lambda_{n_2}v_1\otimes \psi_{n_2},\\
(L_1\otimes 1)(\phi_{m_1}\otimes v_2)=\mu_{m_1} \phi_{m_1}\otimes v_2,
\end{gather*}
where $\lambda_k$ and $\mu_k$ are given in \eqref{eq:eigenvalues}.
Using these eigenvectors, we can define an inner product on $V_1$ and $V_2$ on the basis elements by
\[
\big\langle \phi_{m_1},\phi_{m_1'}\big\rangle_{V_1}=\delta_{m_1m_1'}\qquad \text{and}\qquad \big\langle \psi_{n_2},\psi_{n_2'}\big\rangle_{V_2}=\delta_{n_2n_2'}.
\]
This in turn defines an inner product on $V_1\otimes V_2$ by
\[
\langle v_1\otimes v_2,v_1'\otimes v_2'\rangle_{V_1\otimes V_2}= \langle v_1,v_1'\rangle_{V_1}\langle v_2,v_2'\rangle_{V_2},\qquad v_1,v_1'\in V_1, \quad v_2,v_2'\in V_2.
\]
With respect to this inner product, both $L_1\otimes 1$ and $1\otimes K_1$ are self-adjoint. Also, we can use $\psi_{n_2}$ and $\phi_{m_1}$ to construct eigenvectors of $K_2$ and $L_2$ respectively. Let us start with $K_2$, for which we can find an eigenvector coming from the eigenvector $\psi_{n_2}$ of $K_1$.
\begin{Proposition}
	For each $\psi_{n_2}$, there exists $\xi\in V_1$ such that $\xi\otimes \psi_{n_2}$ is an eigenvector of $K_2$.
\end{Proposition}
\begin{proof}
	Let $\lambda_{n_2}$ be the eigenvalue corresponding to eigenvector $\psi_{n_2}$, i.e.,
	\[
	K_1\psi_{n_2}=\lambda_{n_2}\psi_{n_2}.
	\]
	Since $K_2$ and $1\otimes K_1$ commute, we have for any vector $v\in V_1$,
	\[
	(1\otimes K_1)K_2 (v\otimes \psi_{n_2}) = \lambda_{n_2} K_2 (v\otimes \psi_{n_2}).
	\]
	Thus $K_2 (v\otimes \psi_{n_2})$ is also an eigenvector of $1\otimes K_1$ with eigenvalue $\lambda_{n_2}$. Since the eigenspaces of~$K_1$ in~$V_2$ are one-dimensional, there exists a unique $w\in V_1$ such that
	\[
	K_2 (v\otimes \psi_{n_2}) = w\otimes \psi_{n_2}.
	\]
	Define $A$ to be the operator on $V_1$ that resembles the action of $K_2$ on the left side of the above tensor product, i.e.,
	\[
	A v =w.
	\]
	Since $K_2$ is a linear operator, $A$ is as well. Therefore, $A$ has an eigenvector, say $\xi$ with eigenvalue~$\lambda'$. Then we have
	\[
	K_2 (\xi\otimes \psi_{n_2}) = (A\xi)\otimes \psi_{n_2} =\lambda'\xi\otimes \psi_{n_2}. \tag*{\qed}
	\]\renewcommand{\qed}{}
\end{proof}

Since $K_2$ and $L_1\otimes 1$ satisfy the AW relation \eqref{eq:aw5}, we get a similar `ladder' property as in the original $\aw$. Namely we can use $L_1\otimes 1$ to create a ladder of eigenvectors for $K_2$ from the eigenvector $\xi\otimes\psi_{n_2}$. Moreover, the `Leonard pair' property still holds. That is, $L_1\otimes 1$ acts as a three-term operator on the eigenvectors of $K_2$. The proof is similar to the original $\aw$, where one has to realize that the non-constant term `$1\otimes K_1$' that appears in the AW-relation~\eqref{eq:aw5} must be treated as the constant~$\lambda_{n_2}$: the eigenvalue of how it acts on $\xi\otimes \psi_{n_2}$.
\begin{Proposition}\label{prop:opnieuw3term}
	For each $\psi_{n_2}$, there is a basis for $V_1$ given by
	\[
	\big\{\psi_{n_1(k)}^{n_2}\big\}_{k=0}^{N_1},
	\]
	such that $\psi_{n_1}^{n_2}\otimes\psi_{n_2}$ is an eigenvector of $K_2$ with eigenvalue $\lambda_{n_1}$,
	\begin{gather*}
		K_2\big(\psi_{n_1}^{n_2}\otimes\psi_{n_2}\big) = \lambda_{n_1}\psi_{n_1}^{n_2}\otimes\psi_{n_2}.
	\end{gather*}
	Here,
	\[
	n_1:=n_1(k)=k+n_2-\half N_1.
	\]
	Moreover, $L_1\otimes 1$ acts as a three-term operator on these eigenvectors,
	\begin{gather*}
(L_1\otimes 1)\big(\psi_{n_1}^{n_2}\otimes\psi_{n_2}\big)= a_{n_1}\big(\bm{\alpha}_{L_1}^{n_2}\big) \psi_{n_1-1}^{n_2}\otimes\psi_{n_2} + b_{n_1}\big(\bm{\alpha}_{L_1}^{n_2}\big) \psi_{n_1}^{n_2}\otimes\psi_{n_2} \\
\hphantom{(L_1\otimes 1)\big(\psi_{n_1}^{n_2}\otimes\psi_{n_2}\big)=}{}
			 + a_{n_1+1}\big(\bm{\alpha}_{L_1}^{n_2}\big)\psi_{n_1+1}^{n_2}\otimes\psi_{n_2}, 
	\end{gather*}
	where $\bm{\alpha}_{L_1}^{n_2}=(\alpha_0+2n_2,\alpha_1,\alpha_2,-N_1-1)$ and the functions $a_{n}$ and $b_{n}$ are given by~\eqref{eq:annew} and~\eqref{eq:bnnew}.
\end{Proposition}
\begin{proof}
	The proof is similar to the construction of the representation of the original $\aw$, see~\cite{Zhedanov1991}. Let $\psi_{n_2}$ be an eigenvector of $K_1$. By the previous proposition, there exists an eigenvector $\psi_{p}^{n_2}\otimes \psi_{n_2}$ of $K_2$, say with eigenvalue $\lambda(p)$,
	\begin{gather}
		K_2 \psimn{p}{n_2} = \lambda(p) \psimn{p}{n_2}.\label{eq:1503ev}
	\end{gather}
	Define
	\[
	A=[K_2,L_1\otimes 1]_q,
	\]
	which plays the role of the third generator of the original $\aw$. We will use the ladder property of the Askey--Wilson algebra~\eqref{eq:aw5}. That is, we will show there exist $\alpha,\beta,\gamma\in\CC$ such that
	\begin{gather*}
		\psimn{s}{n_2} = \left(\alpha K_2 + \beta (L_1\otimes 1) + \gamma A\right)\psimn{p}{n_2}
	\end{gather*}
	is also an eigenvector of $K_2$ with a different eigenvalue $\lambda(s)$. Requiring
	\[K_2\psimn{s}{n_2} = \lambda(s) \psimn{s}{n_2}\]
	implies
	\begin{gather*}
		\big( \alpha K_2^2 + \beta K_2(L_1\otimes 1) + \gamma K_2A\big)\psimn{p}{n_2} = \lambda(s) (\alpha K_2 + \beta (L_1\otimes 1) + \gamma A )\psimn{p}{n_2}.
	\end{gather*}
	Using \eqref{eq:1503ev} and the algebra relations \eqref{eq:aw5}, this leads to
	\begin{gather*}
		0= \psimn{p}{n_2}\big[ \alpha \lambda(p)^2 + \beta \big(q^{-2}\lambda(p)(L_1\otimes 1) + q^{-1}A \big) \color{white}q\hat{C}_1\color{black} \\
\hphantom{0=}{} +\gamma \big(q^2\lambda(p)A - q\lambda(p)\hat{B} - q\hat{C}_1(L_1\otimes 1) - q\hat{D}_1\big) - \lambda(s) \alpha \lambda(p) + \beta (L_1\otimes 1) + \gamma A\big].
	\end{gather*}
	Here, $\hat{B}$, $\hat{C_1}$, $\hat{D_1}$ are the parameters of the Askey--Wilson algebra~\eqref{eq:aw5} such that
	\[
	\chq(2) K_2 (L_1\otimes 1) K_2 - K_2^2 (L_1\otimes 1) - (L_1\otimes 1) K_2^2 = \hat{B}K_2 + \hat{C}_1 (L_1\otimes 1) + \hat{D}_1.
	\]
	That is,
	\begin{gather*}
		\hat{B}=(\shq(1))^2 ( A_1 (1\otimes K_1)- A_2 A_{N_1} ),\qquad \hat{C}_1 = -(\chq(1))^2 \sigma_K,\\
		\hat{D}_1= \chq(1\big(A_1A_{N_1}\sigma_K + (\shq(1))^2A_2 (1\otimes K_1)\big).
	\end{gather*}
	Since
	\[
	(1\otimes K_1) \psi_{n_1}^{n_2}\otimes\psi_{n_2} = \lambda_{n_2} \psi_{n_1}^{n_2}\otimes\psi_{n_2},
	\]
	$1\otimes K_1$ acts as the constant $\lambda_{n_2}$ on $\psimn{p}{n_2}$. Therefore, we can interpret $\hat{B}$, $\hat{C}_1$ and $\hat{D}_1$ as constants and proceed as Zhedanov in \cite{Zhedanov1991}. That is, we can find a ladder of eigenvectors of $K_2$
	\begin{gather*}
		\psimn{n_1}{n_2},\qquad n_1= n_2-\half N_1, n_2-\half N_1 +1,\ldots, n_2+\half N_1
	\end{gather*}		
	with eigenvalues $\lambda_{n_1}$. Also, $L_1\otimes 1$ acts as a three-term operator on these eigenvectors,
\begin{gather*}
(L_1\otimes 1)\big(\psi_{n_1}^{n_2}\otimes\psi_{n_2}\big)= a_{n_1} \psi_{n_1-1}^{n_2}\otimes\psi_{n_2} + b_{n_1} \psi_{n_1}^{n_2}\otimes\psi_{n_2} + a_{n_1+1}\psi_{n_1+1}^{n_2}\otimes\psi_{n_2}.
\end{gather*}
Let us find the coefficients $a_{n_1}$ and $b_{n_1}$. Since $1\otimes K_1$ can be interpreted here as
\[
\lambda_{n_2}=\shq(\alpha_0 +2n_2),
\]
	we are in the setting \eqref{eq:Aalpha} where we have to replace $\alpha_0$ by $\alpha_0 + 2n_2$. Therefore, $a_{n_1}$ and $b_{n_1}$ can be found using the formulas~\eqref{eq:annew} and~\eqref{eq:bnnew} with $\bm{\alpha}_{L_1}^{n_2}=(\alpha_0+2n_2,\alpha_1,\alpha_2,-N_1-1)$.
\end{proof}
\begin{Remark}
	Note that the label $n_2$ of the eigenvector of $1\otimes K_1$ gets into the parameter of the three-term recurrence relation for the left side of the eigenvectors for $K_2$. We will see that in the bivariate $q$-Racah polynomials this leads to the degree of the right $q$-Racah polynomial getting into the parameters of the left $q$-Racah polynomial.
\end{Remark}

By symmetry, we get a similar result for $L_2$ and $1\otimes K_1$.
\begin{Corollary}\label{cor:3termomgedraaid}
	For each $\phi_{m_1}$, there is a basis for $V_2$ given by
	\[
	\big\{\phi_{m_2(j)}^{m_1}\big\}_{j=0}^{N_2},
	\]
	such that $\phi_{m_1}\otimes\phi_{m_2}^{m_1}$ is an eigenvector of $L_2$ with eigenvalue $\mu_{m_2}$,
	\begin{gather*}
		L_2\big(\phi_{m_1}\otimes\phi_{m_2}^{m_1}\big) = \mu_{m_2}\phi_{m_1}\otimes\phi_{m_2}^{m_1}.
	\end{gather*}
	Here we have
	\[
	m_2:=m_2(j)=j+m_1-\half N_2.
	\]
	Moreover, $1\otimes K_1$ acts as a three-term operator on these eigenvectors,
	\begin{gather*}
(1\otimes K_1)\big(\phi_{m_1}\otimes\phi_{m_2}^{m_1}\big)= a_{m_2}\big(\bm{\alpha}_{K_1}^{m_1}\big) \phi_{m_1}\otimes\phi_{m_2-1}^{m_1} + b_{m_2}\big(\bm{\alpha}_{K_1}^{m_1}\big) \phi_{m_1}\otimes\phi_{m_2}^{m_1} \\
\hphantom{(1\otimes K_1)\big(\phi_{m_1}\otimes\phi_{m_2}^{m_1}\big)=}{}
+ a_{m_2+1}\big(\bm{\alpha}_{K_1}^{m_1}\big)\phi_{m_1}\otimes\phi_{m_2+1}^{m_1} ,
	\end{gather*}
	where $\bm{\alpha}_{K_1}^{m_1}=(\alpha_1+2m_1,\alpha_0,\alpha_2,-N_2-1)$ and the functions $a_{m}$ and $b_{m}$ are again given by \eqref{eq:annew} and~\eqref{eq:bnnew}.
\end{Corollary}
\begin{proof}
From \eqref{eq:nieuweA} and \eqref{eq:Aalpha}, we see that interchanging the roles of $K$ and $L$ is the same as interchanging $\alpha_0\leftrightarrow \alpha_1$ and $\sigma_L\leftrightarrow \sigma_K$. Therefore, this result is just Proposition~\ref{prop:opnieuw3term} together with these substitutions.
\end{proof}

We are now ready to prove the main theorem of this section. Define the overlap coefficients
\begin{gather}
	P_{n_1,n_2}(m_1,m_2)=C(n_1,n_2,m_1,m_2)\big\langle \phi_{m_1}\otimes \phi_{m_2}^{m_1},\psi_{n_1}^{n_2}\otimes \psi_{n_2}\big\rangle_{V_1\otimes V_2},\label{eq:defbivarqracah}
\end{gather}
where $C$ is the normalizing function defined by
\begin{align}
	C(n_1,n_2,m_1,m_2)=\frac{\big\langle\phi_{m_1(0)},\psi_{n_1(0)}^{n_2}\big\rangle_{V_1} \big\langle\phi_{m_2(0)}^{m_1}, \psi_{n_2(0)}\big\rangle_{V_2}}{\big\langle\phi_{m_1},\psi_{n_1(0)}^{n_2}\big\rangle_{V_1}\big\langle \phi_{m_2}^{m_1},\psi_{n_2(0)}\big\rangle_{V_2}\big\langle\phi_{m_1(0)},\psi_{n_1}^{n_2}\big\rangle_{V_1} \big\langle\phi_{m_2(0)}^{m_1},\psi_{n_2}\big\rangle_{V_2}}.\label{eq:normalizingbivar}
\end{align}
Then we can show that $P_{n_1,n_2}(m_1,m_2)$ are bivariate $q$-Racah polynomials. This comes from the observation that just as in the univariate case, $\big\langle\phi_{m_1},\psi_{n_1}^{n_2}\big\rangle$ and $\big\langle\phi_{m_2}^{m_1},\psi_{n_2}\big\rangle$ are $q$-Racah polynomials with parameters $\bm{\alpha}^{n_2}_{L_1}$ and $\bm{\alpha}_{K_1}^{m_2}$ respectively instead of $\bm{\alpha}=(\alpha_0,\alpha_1,\alpha_2,\alpha_3)$ in \eqref{eq:qracahparnew}. There is essentially only one difference in the proof of the bivariate case compared to the univariate case. The overlap coefficient $\big\langle\phi_{m_2}^{m_1},\psi_{n_2}\big\rangle$ is a $q$-Racah polynomial of degree $j_2$ in the variable~$k_2$, i.e., the degree and variable have switched. We can interchange these by a simple change in parameters.
\begin{Theorem}
	We have
	\begin{gather*}
		P_{n_1(k_1),n_2(k_2)}(\mu_{m_1(j_1)},\mu_{m_2(j_2)})=R_{k_1,k_2}\big(y_{j_1},y_{j_2};\alpha_0,\alpha_1,\alpha_2,N_1,N_2;q^2\big),
	\end{gather*}
	where the right-hand side is a product of two $q$-Racah polynomials given by
	\begin{gather}
R_{k_1}\big(y_{j_1};-q^{\alpha_0+2k_2-N_2+\alpha_1+\alpha_2-N_1-1}, q^{\alpha_0+2k_2-N_2-\alpha_1-\alpha_2-N_1-1}, q^{-2N_1-2},-q^{2\alpha_1};q^2\big)\nonumber\\
\qquad{} \times R_{k_2}\big(y_{j_2};-q^{\alpha_0+\alpha_1+2j_1-N_1-\alpha_2-N_2-1}, q^{\alpha_0-\alpha_1-2j_1+N_1+\alpha_2-N_2-1},\nonumber\\
\qquad\quad \ {} q^{-2N_2-2},-q^{2\alpha_1+4j_1-2N_1};q^2\big),
 \label{eq:bivarqracahmijn}
	\end{gather}
	with $y_{j_1}=q^{-2j_1}-q^{2\alpha_1-2N_1+2j_j}$ and $y_{j_2} = q^{-2j_2}-q^{2\alpha_1-2N_1-2N_2+4j_1+2j_2}$.
\end{Theorem}
\begin{proof}
	The proof is similar to the original $\aw$, we only have to do it twice. It is convenient to split $P_{n_1,n_2}(m_1,m_2)$ into the two parts concerning $V_1$ and $V_2$ and analyze both parts separately. We have
	\begin{gather}
		P_{n_1,n_2}(m_1,m_2) = \frac{\big\langle \phi_{m_1},\psi_{n_1}^{n_2}\big\rangle _{V_1}\big\langle \phi_{m_1(0)},\psi_{n_1(0)}^{n_2}\big\rangle _{V_1}}{\big\langle \phi_{m_1},\psi_{n_1(0)}^{n_2}\big\rangle _{V_1}\big\langle \phi_{m_1(0)},\psi_{n_1}^{n_2}\big\rangle _{V_1}} \times\frac{\big\langle \phi_{m_2}^{m_1}, \psi_{n_2}\big\rangle _{ V_2}\big\langle \phi_{m_2(0)}^{m_1}, \psi_{n_2(0)}\big\rangle _{ V_2}}{\big\langle \phi_{m_2}^{m_1}, \psi_{n_2(0)}\big\rangle _{ V_2}\big\langle \phi_{m_2(0)}^{m_1}, \psi_{n_2}\big\rangle _{ V_2}}.\label{eq:opsplitsen}
	\end{gather}
	For notational convenience, we will define $q$-Racah polynomials in terms of $\bm{\alpha}=(\alpha_0,\alpha_1,\alpha_2,\alpha_3)$,
	\[
	r_{k}(j;\bm{\alpha})=R_k\big(y_j;-q^{\alpha_0-\alpha_1+\alpha_2+\alpha_3},q^{\alpha_0-\alpha_1-\alpha_2+\alpha_3},q^{2\alpha_3},-q^{-2\alpha_1};q^2\big).
	\]
	We will show that the part concerning $V_1$ of \eqref{eq:opsplitsen} is equal to
	\begin{gather}
P_{n_1(k_1)}^{V_1}(m_1(j_1);n_2):=\frac{\big\langle \phi_{m_1(j_1)},\psi_{n_1(k_1)}^{n_2}\big\rangle _{V_1}\big\langle \phi_{m_1(0)},\psi_{n_1(0)}^{n_2}\big\rangle _{V_1}}{\big\langle \phi_{m_1(j_1)},\psi_{n_1(0)}^{n_2}\big\rangle _{V_1}\big\langle \phi_{m_1(0)},\psi_{n_1(k_1)}^{n_2}\big\rangle _{V_1}} = r_{k_1}\big(j_1;\bm{\alpha}_{L_1}^{n_2}\big),\label{eq:polracah1}
	\end{gather}
	where $\bm{\alpha}_{L_1}^{n_2}=(\alpha_0+2n_2,\alpha_1,\alpha_2,-N_1-1)$.
	For the part concerning $V_2$ of~\eqref{eq:opsplitsen} we have
	\begin{align}
P_{n_2(k_2)}^{V_2}(m_2(j_2);m_1):={}&\frac{\big\langle \phi_{m_2(j_2)}^{m_1},\psi_{n_2(k_2)}\big\rangle _{ V_2}\big\langle \phi_{m_2(0)}^{m_1}, \psi_{n_2(0)}\big\rangle _{ V_2}}{\big\langle \phi_{m_2(j_2)}^{m_1}, \psi_{n_2(0)}\big\rangle _{ V_2}\big\langle \phi_{m_2(0)}^{m_1}, \psi_{n_2(k_2)}\big\rangle _{ V_2}}\nonumber\\
={} &r_{k_2}(j_2;\alpha_0,\alpha_1+2m_1,-\alpha_2,-N_2-1), \label{eq:polracah2}
	\end{align}
	which comes from $\bm{\alpha}_{K_1}^{n_1}$ where $\alpha_0 \leftrightarrow \alpha_1+2m_1$ and $\alpha_2\leftrightarrow -\alpha_2$, which is necessary to interchange degree and variable of the $q$-Racah polynomial.
	
	Note that
	\[
	P_{n_1(0)}^{V_1}(m_1;n_2)=1=P_{n_1}^{V_1}(m_1(0);n_2) \qquad \text{and}\qquad P_{n_2(0)}^{V_2}(m_2;m_1)=1=P_{n_2}^{V_2}(m_2(0);m_1).
	\]
	Let us start with $P_{n_1}^{V_1}(m_1;n_2)$. We know that the self adjoint operator $L_1\otimes 1$ acts as multiplication by $\mu_{m_1}$ on $\phi_{m_1}\otimes \phi_{m_2}^{m_1}$ and as a three-term operator (Proposition~\ref{prop:opnieuw3term}) on $\psi_{n_1}^{n_2}\otimes \psi_{n_2}$. It only acts non-trivially on $V_1$, thus we can look at the action on the left side of the inner product. Entirely similar to~\eqref{eq:origin3term} we get, using a slight abuse of notation,
	\begin{gather*}
		\mu_{m_1} \inprod{\phi_{m_1}}{\psi_{n_1}^{n_2}}_{V_1}= \big\langle L_1\phi_{m_1},\psi_{n_1}^{n_2}\big\rangle _{V_1}
		= \big\langle \phi_{m_1},L_1\psi_{n_1}^{n_2}\big\rangle _{V_1} \\
\hphantom{\mu_{m_1} \inprod{\phi_{m_1}}{\psi_{n_1}^{n_2}}_{V_1}}{}
= a_{n_1}(\bm{\alpha}_{L_1}^{n_2}) \big\langle \phi_{m_1},\psi_{n_1-1}^{n_2}\big\rangle _{V_1}+b_{n_1}(\bm{\alpha}_{L_1}^{n_2})\big\langle \phi_{m_1},\psi_{n_1}^{n_2}\big\rangle _{V_1}\\
\hphantom{\mu_{m_1} \inprod{\phi_{m_1}}{\psi_{n_1}^{n_2}}_{V_1}=}{}+a_{n_1+1}(\bm{\alpha}_{L_1}^{n_2})\big\langle \phi_{m_1},\psi_{n_1+1}^{n_2}\big\rangle _{V_1},
	\end{gather*}
	where $\bm{\alpha}_{L_1}^{n_2}$ can be found in Proposition~\ref{prop:opnieuw3term} and $a_{n_1}(\bm{\alpha})$, $b_{n_1}(\bm{\alpha})$ are given by~\eqref{eq:annew}, \eqref{eq:bnnew}. We want to obtain a three-term recurrence relation for $P_{n_1}^{n_2}(m_1)$ from this equation. However, this will not work because of the appearance of~$n_1$ in the normalizing factor $\big\langle \phi_{m_1(0)},\psi_{n_1}^{n_2}\big\rangle _{V_1}$. To solve this, we define
	\[
	p_{n_1}^{V_1}(m_1;n_2):=\frac{\big\langle \phi_{m_1},\psi_{n_1}^{n_2}\big\rangle _{V_1}}{\big\langle \phi_{m_1},\psi_{n_1(0)}^{n_2}\big\rangle _{V_1}},
	\]
	which is only the left part of $P_{n_1}^{V_1}(m_1;n_2)$ in \eqref{eq:polracah1}. Now we have
	\begin{gather*}
		\mu_{m_1}p_{n_1}^{V_1}(m_1;n_2)= a_{n_1}\big(\bm{\alpha}_{L_1}^{n_2}\big)p_{n_1-1}^{V_1}(m_1;n_2) + b_{n_1}\big(\bm{\alpha}_{L_1}^{n_2}\big)p_{n_1}^{V_1}(m_1;n_2)\\
\hphantom{\mu_{m_1}p_{n_1}^{V_1}(m_1;n_2)=}{}
+ a_{n_1+1}\big(\bm{\alpha}_{L_1}^{n_2}\big)p_{n_1+1}^{V_1}(m_1;n_2).
	\end{gather*}
	Together with the initial conditions
	\[
	p_{n_1(0)}^{V_1}(m_1;n_2)=1 \qquad \text{and}\qquad p_{n_1(-1)}^{V_1}(m_1;n_2)=0,
	\]
	this three-term recurrence relation generates a set $\big\{p_{n_1(k_1)}^{V_1}(m_1;n_2)\big\}_{k_1=0}^{N_1}$ of polynomials in the variable $\mu_{m_1}$. Similarly as in the original $\aw$, one can show that this coincides with the recurrence relation for $q$-Racah polynomials with parameters $\bm{\alpha}_{L_1}^{n_2}$. Therefore,
	\begin{gather}
		p_{n_1(k_1)}^{V_1}(m_1(j_1);n_2)=h(n_1,n_2)r_{k_1}\big(j_1;\bm{\alpha}_{L_1}^{n_2}\big),\label{eq:3termqracah1}
	\end{gather}
	where $h$ is some normalization function. To prove \eqref{eq:polracah1}, we have to show that $1/h$ is exactly the right-hand side of $P_{n_1}^{V_1}(m_1;n_2)$, i.e.,
	\begin{gather}
		\frac{1}{h(n_1,n_2)}=\frac{\big\langle \phi_{m_1(0)},\psi_{n_1(0)}^{n_2}\big\rangle _{V_1}}{\big\langle \phi_{m_1(0)},\psi_{n_1}^{n_2}\big\rangle _{V_1}}.\label{eq:claimbivarqracah}
	\end{gather}
	Since $P_{n_1}^{V_1}(m_1(0);n_2)=1$, we can use \eqref{eq:3termqracah1} to obtain
	\begin{gather*}
		1=p_{n_1(k_1)}^{V_1}(m_1(0);n_2)\frac{\big\langle \phi_{m_1(0)},\psi_{n_1(0)}^{n_2}\big\rangle _{V_1}}{\big\langle \phi_{m_1(0)},\psi_{n_1}^{n_2}\big\rangle _{V_1}} = h(n_1,n_2)r_{k_1}\big(0;\bm{\alpha}_{L_1}^{n_2}\big)\frac{\big\langle \phi_{m_1(0)},\psi_{n_1(0)}^{n_2}\big\rangle _{V_1}}{\big\langle \phi_{m_1(0)},\psi_{n_1}^{n_2}\big\rangle _{V_1}}.
	\end{gather*}
	By definition of the $q$-hypergeometric series, we also have
	\[r_{k_1}\big(0;\bm{\alpha}_{L_1}^{n_2}\big)=1,\]
	which proves our claim \eqref{eq:claimbivarqracah}.
	
	Let us now turn to $P_{n_2}^{V_2}(m_2;m_1)$, which is slightly more subtle because we do not know how $1\otimes L_1$ acts on $\phi_{m_1}\otimes\phi_{m_2}^{m_1}$. However, since we do know that $1\otimes K_1$ acts as a three-term operator on $\phi_{m_1}\otimes\phi_{m_2}^{m_1}$, we get a three-term recurrence relation in the variable $\mu_{m_2}$ instead of the degree~$n_2$. That is, the roles of `$K$' and `$L$' have interchanged: we use $1\otimes K_1$ as multiplication and three-term operator instead of $L_1\otimes 1$. Using the symmetry of~$\AW{2}$ we can see that the overlap coefficients of $\phi_{m_2}^{m_1}$ and $\psi_{n_2}$ are $q$-Racah polynomials of degree $m_2$ in the variable $n_2$.
	
	That is, Corollary \ref{cor:3termomgedraaid} tells us that
	\begin{gather*}
		\lambda_{n_2} \big\langle \phi_{m_2}^{m_1},\psi_{n_2}\big\rangle _{V_2} = \big\langle \phi_{m_2}^{m_1},K_1\psi_{n_2}\big\rangle _{V_2} = \big\langle K_1\phi_{m_2}^{m_1},\psi_{n_2}\big\rangle _{V_2} \\
		\qquad{} =a_{m_2}\big(\bm{\alpha}_{K_1}^{m_1}\big)\big\langle \phi_{m_2-1}^{m_1},\psi_{n_2}\big\rangle _{V_2}+b_{m_2}\big(\bm{\alpha}_{K_1}^{m_1}\big) \big\langle \phi_{m_2}^{m_1},\psi_{n_2}\big\rangle _{V_2}+a_{m_2+1}\big(\bm{\alpha}_{K_1}^{m_1}\big)\big\langle \phi_{m_2+1}^{m_1},\psi_{n_2}\big\rangle _{V_2},
	\end{gather*}
	Therefore,
	\[
	\frac{\big\langle \phi_{m_2(j_2)}^{m_1}, \psi_{n_2(k_2)}\big\rangle _{ V_2}}{\big\langle \phi_{m_2(0)}^{m_1}, \psi_{n_2(k_2)}\big\rangle _{ V_2}}=g(m_1,m_2(j_2))r_{j_2}\big(k_2;\bm{\alpha}_{K_1}^{m_1}\big),
	\]
for some normalizing function $g$. Using the duality property of $q$-Racah polynomials, which can directly be seen from the $q$-hypergeometric series \eqref{eq:defqracah}, we can swap degree and variable by interchanging $\alpha_0\leftrightarrow\alpha_1$ and $\alpha_2\leftrightarrow-\alpha_2$. In terms of our original definition of the $q$-Racah polynomials \eqref{eq:defqracah}, this interchanges $\alpha\beta\leftrightarrow\gamma\delta$ and $\beta\delta\leftrightarrow\alpha$, while keeping $\gamma$ the same. Therefore,
	\[
	\frac{\big\langle \phi_{m_2(j_1)}^{m_1}, \psi_{n_2(k_2)}\big\rangle _{ V_2}}{\big\langle \phi_{m_2(0)}^{m_1}, \psi_{n_2(k_2)}\big\rangle _{ V_2}}=g(m_1,m_2)r_{k_2}(j_2,\bm{\beta}),
	\]
	where $\bm{\beta}=(\alpha_0,\alpha_1+2m_1,-\alpha_2,-N_2-1)$. Consequently,
	\begin{align}
		P_{n_2(k_2)}^{V_2}(m_2(j_2);m_1)& = \frac{\big\langle \phi_{m_2(j_2)}^{m_1},\psi_{n_2(k_2)}\big\rangle _{ V_2}}{\big\langle \phi_{m_2(0)}^{m_1}, \psi_{n_2(k_2)}\big\rangle _{ V_2}}\frac{\big\langle \phi_{m_2(0)}^{m_1}, \psi_{n_2(0)}\big\rangle _{ V_2}}{\big\langle \phi_{m_2(j_2)}^{m_1}, \psi_{n_2(0)}\big\rangle _{ V_2}} \nonumber\\
		&= g(m_1,m_2)r_{k_2}(j_2,\bm{\beta})\frac{\big\langle \phi_{m_2(0)}^{m_1}, \psi_{n_2(0)}\big\rangle _{ V_2}}{\big\langle \phi_{m_2(j_2)}^{m_1}, \psi_{n_2(0)}\big\rangle _{ V_2}}.\label{eq:pn2m1mu2}
	\end{align}
	Using that
	\[
	1=P_{n_2(0)}^{V_2}(m_2(j_2);m_1)= g(m_1,m_2)r_{0}(j_2;\bm{\beta})\frac{\big\langle \phi_{m_2(0)}^{m_1}, \psi_{n_2(0)}\big\rangle _{ V_2}}{\big\langle \phi_{m_2(j_2)}^{m_1}, \psi_{n_2(0)}\big\rangle _{ V_2}}
	\]
	and
	\[
	r_{0}(j_2,\bm{\beta})=1,
	\]
	we get
	\[
	\frac{g(m_1,m_2)}{\big\langle \phi_{m_2}^{m_1}, \psi_{n_2(0)}\big\rangle _{ V_2}}=\frac{1}{\big\langle \phi_{m_2(0)}^{m_1}, \psi_{n_2(0)}\big\rangle _{ V_2}}.
	\]
	Substituting this into \eqref{eq:pn2m1mu2} gives
	\[
	P_{n_2(k_2)}^{V_2}(m_2(j_2);m_1)= r_{k_2}(j_2,\bm{\beta}) ,
	\]
	which proves the theorem.
\end{proof}

\begin{Remark}
	Note that the bivariate $q$-Racah polynomials in the theorem above are formally not polynomials but rational functions. This is because of the appearance of $j_1$ in the parameters of the right $q$-Racah polynomial $R_{k_2}$. This can be solved easily by putting an appropriate factor in front.
\end{Remark}
\begin{Remark}
	Orthogonality from the bivariate $q$-Racah polynomials arises similarly as in the univariate case. Let
	\[
	\mathbf{N}=\{(l_1,l_2)\colon l_1=0,\dots ,N_1\text{ and }l_2=0,\dots,N_2\}.
	\]
	From $\big\{\phi_{m_1(j_1)}\otimes \phi_{m_2(j_2)}^{m_1(j_1)}\big\}_{(j_1,j_2)\in\mathbf{N}}$ and $\big\{\psi_{n_1(k_1)}^{n_2(k_2)}\otimes \psi_{n_2(k_2)}\big\}_{(k_1,k_2)\in\mathbf{N}}$ both being orthonormal, we obtain
	\begin{gather}
		\delta_{k_1,k_1'}\delta_{k_2,k_2'} =\delta_{n_1(k_1),n_1(k_1')}\delta_{n_2(k_2),n_2(k_2')}= \big\langle \psi_{n_1(k_1')}^{n_2(k_2')},\psi_{n_1(k_1)}^{n_2(k_2)}\big\rangle _{V_1}\big\langle \psi_{n_2(k_2')},\psi_{n_2(k_2)}\big\rangle _{V_2}.\label{eq:orthobivar1}
	\end{gather}
	Using that
	\begin{gather*}
		\psi_{n_1(k_1')}^{n_2(k_2')}=\sum_{j_1=0}^{N_1}\big\langle \psi_{n_1(k_1')}^{n_2(k_2')},\phi_{m_1(j_1)}\big\rangle \phi_{m_1(j_1)}
	\end{gather*}
	and
	\begin{gather*}
		\psi_{n_2(k_2')}=\sum_{j_2=0}^{N_2}\big\langle \psi_{n_2(k_2')},\phi_{m_2(j_2)}^{m_1(j_1)}\big\rangle \phi_{m_2(j_2)}^{m_1(j_1)},
	\end{gather*}
	we obtain that \eqref{eq:orthobivar1} is equal to		
	\begin{gather*}
 \sum_{j_1=0}^{N_1}\sum_{j_2=0}^{N_2} \big\langle \psi_{n_1(k_1')}^{n_2(k_2')},\phi_{m_1(j_1)}\big\rangle _{V_1}\big\langle \phi_{m_1(j_1)},\psi_{n_1(k_1)}^{n_2(k_2)}\big\rangle _{V_1}\big\langle \psi_{n_2(k_2')}, \phi_{m_2(j_2)}^{m_1(j_1)}\big\rangle _{V_2}\big\langle \phi_{m_2(j_2)}^{m_1(j_1)},\psi_{n_2(k_2)}\big\rangle _{V_2}.
\end{gather*}
Rearranging this expression gives
\begin{gather*}
\sum_{j_1=0}^{N_1}\sum_{j_2=0}^{N_2} \big\langle \phi_{m_1(j_1)}\otimes \phi_{m_2(j_2)}^{m_1(j_1)},\psi_{n_1(k_1)}^{n_2(k_2)}\otimes \psi_{n_2(k_2)}\big\rangle _{V_1\otimes V_2}\\
\qquad{} \times \overline{\big\langle \phi_{m_1(j_1)}\otimes \phi_{m_2(j_2)}^{m_1(j_1)},\psi_{n_1(k_1')}^{n_2(k_2')}\otimes \psi_{n_2(k_2')}\big\rangle }_{V_1\otimes V_2}.
	\end{gather*}
	Therefore, using \eqref{eq:defbivarqracah}, we obtain the orthogonality
	\begin{gather*}
		\sum_{j_1=0}^{N_1}\sum_{j_2=0}^{N_2} w_2(j_1,j_2,k_1,k_2) P_{n_1(k_1),n_2(k_2)}(m_1(j_1),m_2(j_2))\overline{P_{n_1(k_1'),n_2(k_2')}(m_1(j_1),m_2(j_2))} \\
		\qquad{} =\delta_{k_1,k_1'}\delta_{k_2,k_2'} ,
	\end{gather*}
	where the weight function is given by
	\[
	w_2(j_1,j_2,k_1,k_2)=\frac{1}{|C(n_1(k_1),n_2(k_2),m_1(j_1),m_2(j_2)|^2}.
	\]
	The normalizing function $C$ can be found in \eqref{eq:normalizingbivar}. This weight function is just a product of the weight functions of the two univariate $q$-Racah polynomials in \eqref{eq:bivarqracahmijn},
	\begin{align}
		w_2(j_1,j_2,k_1,k_2):={}& w_2(j_1,j_2,k_1,k_2,\alpha_0,\alpha_1,\alpha_2,\alpha_3;q)\nonumber\\
={}& w(j_1,k_1,\bm{\alpha}_{L_1}^{n_2(k_2)};q)w(j_2,k_2,\bm{\beta};q),\label{eq:doubleweight}
	\end{align}	
	where $w$ is the weight function for $q$-Racah polynomials given in \eqref{eq:weightqracahnew}, $\bm{\alpha}_{L_1}^{n_2(k_2)}$ is given in Proposition \ref{prop:opnieuw3term} and $\bm{\beta}=(\alpha_0,\alpha_1+2m_1(j_1),-\alpha_2,-N_2-1)$. This can be seen from analyzing both $q$-Racah polynomials \eqref{eq:polracah1} and \eqref{eq:polracah2} from $P_{n_1,n_2}(m_1,m_2)$. Let us show this for $P_{n_1(k_1)}^{V_1}(m_1,m_2)$ in \eqref{eq:polracah1}. Entirely similar to the univariate case \eqref{eq:orthogonalityqracah} we have
	\begin{align*}
		\delta_{k_1,k_1'}&=\big\langle \psi_{n_1(k_1')}^{n_2(k_2)},\psi_{n_1(k_1)}^{n_2(k_2)}\big\rangle _{V_1}=\sum_{j_1=0}^{N_1} \big\langle \psi_{n_1(k_1')}^{n_2(k_2)},\phi_{m_1(j_1)}\big\rangle _{V_1}\big\langle \phi_{m_1(j_1)},\psi_{n_1(k_1)}^{n_2(k_2)}\big\rangle _{V_1} \\
		&= \sum_{j_1=0}^{N_1} w(j_1,k_1,\bm{\alpha}_{L_1}^{n_2(k_2)};q)P_{n_1(k_1)}^{V_1}(m_1(j_1);m_2)\overline{P_{n_1(k_1')}^{V_1}(m_1(j_1);m_2)}.
	\end{align*}
	Similarly, one finds that
	\begin{gather*}
		\delta_{k_2,k_2'}=\sum_{j_2=0}^{N_2} w(j_2,k_2,\bm{\beta};q)P_{n_2(k_2)}^{V_2}(m_2(j_2);n_1)\overline{P_{n_2(k_2')}^{V_2}(m_2(j_2);n_1)},
	\end{gather*}
	which proves \eqref{eq:doubleweight}.
\end{Remark}
\begin{Remark}
	The bivariate $q$-Racah polynomials~\eqref{eq:bivarqracahmijn} are the same, up to a factor in front of the two $q$-hypergeometric series, as found by Gasper and Rahman \cite{Gasper2007} after a change of parameters, variables and degrees. In the notation of \cite[equation~(2.6)]{Gasper2007}, looking only at the $q$-hypergeometric series we have
	\begin{gather}
R_{n_1,n_2}(x_1,x_2;a_1,a_2,a_3,b,N;q)\nonumber\\
\qquad{} =r_{n_1}\big(x_1;b,a_2/q,a_1q^{2x_2},x_2;q\big)r_{n_2}\big(x_2-n_1;ba_2q^{n_1},a_3/q,a_1a_2q^{N+n_1},N-n_1;q\big)\nonumber\\
\qquad{} = \qhyp{4}{3}{q^{-n_1},\ ba_2q^{n_1},\ q^{-x_1},\ a_1q^{x_1}}{bq,\ a_1a_2q^{x_2},\ q^{-x_2}}{q;q}\nonumber\\
\qquad\quad \ {}\times \qhyp{4}{3}{q^{-n_2},\ ba_2a_3q^{2n_1+n_2},\ q^{n_1-x_2},\ a_1a_2q^{n_1+x_2}}{ba_2q^{2n_1+1},\ a_1a_2a_3q^{n_1+N},\ q^{n_1-N}}{q;q}.\label{eq:qracahgasrah}
	\end{gather}
Let us now make a change of parameters, variables and degrees. First, define
\[
	\alpha=\alpha_0+\alpha_1+\alpha_2\qquad \text{and}\qquad \hat{\alpha}=\alpha_0-\alpha_1+\alpha_2.
	\]
Then take parameters
	\begin{gather*}
	a_1=-q^{2\alpha_0+2\alpha_2+2},\qquad a_2=q^{-2N_2},\qquad a_3=q^{-2N_1},\qquad b=-q^{2\alpha_0},\\ 2N=-\hat{\alpha}+N_1+N_2-1,
	\end{gather*}
	variables
	\[
	2x_1=2j_1+2j_2-\hat{\alpha}-N_1-N_2-1,\qquad 2x_2=2j_1-\hat{\alpha}-N_1+N_2-1,
	\]
	and degrees $n_1=k_2$ and $n_2=k_1$. If we also take $q^2$ instead of $q$, we obtain that \eqref{eq:qracahgasrah} is equal to
	\begin{gather}
		\qhyp{4}{3}{q^{-2k_2},\, -q^{2k_2+2\alpha_0-2N_2},\, q^{-2j_1-2j_2+\hat{\alpha}+N_1+N_2+1},\, -q^{2j_1+2j_2+\alpha-N_1-N_2+1}}{-q^{2\alpha_0+2},\, -q^{2j_1+\alpha-N_1-N_2-1},\, q^{-2j_1+\hat{\alpha}+N_1-N_2+1}}{q^2;q^2}\label{eq:bivarqracah1}\\
		\times\qhyp{4}{3}{q^{-2k_1},\, -q^{4k_2+2k_1+2\alpha_0-2N_1-2N_2},\, q^{2k_2-2j_1+\hat{\alpha}+N_1-N_2+1},\, -q^{2k_2+2j_1+\alpha-N_1-N_2-1}}{-q^{4k_2+2\alpha_0-2N_2+2},\, -q^{2k_2+\alpha-N_1-N_2-1},\, q^{2k_2+\hat{\alpha}-N_1-N_2+1}}{q^2;q^2}\!.\nonumber
	\end{gather}
	Then apply to both $q$-Racah polynomials \eqref{eq:bivarqracah1} 
Sears' transformation formula for a~balanced~${}_{4}\phi_{3}$ \cite[equation~(III.15)]{GR2004},
	\[\qhyp{4}{3}{q^{-n},\ \alpha,\ \beta,\ \gamma}{\delta,\ \varepsilon,\ \zeta}{q;q}=\frac{\big(\alpha^{-1}\delta,\alpha^{-1}\zeta;q\big)_n}{(\delta,\zeta;q)_n}\qhyp{4}{3}{q^{-n},\ \alpha,\ \beta^{-1}\varepsilon,\ \gamma^{-1}\varepsilon}{\alpha\delta^{-1}q^{1-n},\ \epsilon,\ \alpha\zeta^{-1}q^{1-n}}{q;q}.\]
	Afer this transformation, our product of $q$-Racah polynomials \eqref{eq:bivarqracahmijn} is the same as the bivariate $q$-Racah polynomial defined by Gasper and Rahman \cite{Gasper2007} up to a factor in front of the $q$-hypergeometric series.
\end{Remark}

\section[AW\_2 and q-difference operators for bivariate q-Racah polynomials]{$\boldsymbol{\AW{2}}$ and $\boldsymbol{q}$-difference operators\\ for bivariate $\boldsymbol{q}$-Racah polynomials}\label{sec:qdif}

Until now, the structure parameters of an AW-relation could be interpreted as a constant. Either it was a central element of $\AW{2}$ or it acted as multiplication by an eigenvalue. In this way, one generator of the AW-relation worked as a three-term operator on the eigenvectors of the other generator. In this section, we will see that this is no longer true for all generators that satisfy an AW-relation. We will show that $L_2$ acts as a 9-term operator on the eigenvectors $\psi_{n_1}^{n_2}\otimes \psi_{n_2}$. Crucial here is the following observation. Suppose we have elements $K,L$ and $(A_k)_{k=0}^3$ in an algebra and $A_4,A_5\in\RR$ such that
\[
\awr(K,L\mid A_0,A_1,A_2,A_3,A_4,A_5)=0
\]
and $(A_k)_{k=0}^3$ are locally central. Then we will show that $K$ still acts as a three-term operator on the eigenspaces of $L$. Also, the form of the spectrum of $K$ and $L$ is similar to what we have seen before. This result is proven in the following lemma.
\begin{Lemma}\label{lem:threetermeigenspace}
	Let $\A$ be an algebra, $K$, $L$ and $(A_k)_{k=0}^3$ elements in $\A$ and $A_4,A_5\in\RR$. Suppose
	\[
	\awr(K,L\mid A_0,A_1,A_2,A_3,A_4,A_5)=0
	\]
	and $(A_k)_{k=0}^3$ are locally central. Then we have the following:
\begin{enumerate}\itemsep=0pt
		\item[$(i)$] The spectra of $K$ and $L$ have the same hyperbolic form as in the original $\aw$. In particular, if $A_4=A_5=(\shq(1))^2$ we have
		\begin{gather*}
			\lambda_n(p_0)=\shq(2n+p_0+1)
		\end{gather*}
		for some $p_0\in\RR$.
		\item[$(ii)$] Denote the eigenspace corresponding to $\lambda_n(p_0)$ of $K$ by $\Psi(n)$. Then $L$ acts as a three-term operator on the eigenspaces of $K$. That is,
		\[
		L \Psi(n) \subset \Psi(n-1)\cup \Psi(n)\cup\Psi(n+1).
		\]
		Here the eigenspace $\Psi(n)$ is defined as subspace of vectors which have eigenvalue $\lambda_n(p_0)$ and
		\[
		L \Psi(n)=\{ Lv\colon v\in\Psi(n)\}.
		\]
	\end{enumerate}
\end{Lemma}
\begin{proof}
	We proceed similarly as in the original algebra $\aw$ and in the proof of Proposition~\ref{prop:opnieuw3term}, only with a slight adjustment, since in the current setting $(A_k)_{k=0}^3$ cannot all be interpreted as constants. Let $B$, $C_0$, $C_1$, $D_0$, $D_1$ as in~\eqref{eq:nieuweA} and let $\psi_n$ be an eigenvector of $K$. We want to find a new eigenvector $\psi_s$ of $K$. Take
	\begin{gather*}
		\psi_s = (\beta L + \gamma [K,L]_q + \delta B + \varepsilon D_1)\psi_n
	\end{gather*}
	for some constants $\beta$, $\gamma$, $\delta$, $\varepsilon$ that still need to be determined. Requiring $\psi_s$ to be an eigenvector of $K$ with a different eigenvalue leads to
	\begin{gather}
		\lambda_n^2+\ls^2-\chq(2)\lambda_n\ls +C_1=0\label{eq:meervanhetzelfde}
	\end{gather}
	and
	\begin{align}
		-q\gamma(\lambda_n B +D_1) + (\lambda_n-\lambda_s)\delta B +\varepsilon(\lambda_n-\lambda_s) D_1 =0.\nonumber
	\end{align}
	This second equation is equivalent to
	\begin{align}
		(\delta(\lambda_n-\lambda_s)-q\gamma\lambda_n)B + (\varepsilon(\lambda_n-\lambda_s)-q\gamma) D_1 =0, \label{eq:nieuwevergaw}
	\end{align}
	which has a nontrivial solution for any nonzero $\gamma$ and $\lambda_s\neq\lambda_n$. Equation \eqref{eq:meervanhetzelfde} is the same as in the original $\aw$ and forces the spectrum of $K$ to be of the same hyperbolic form, proving~(i).

For (ii), let $\psi_n\in\Psi(n)$. Since \eqref{eq:nieuwevergaw} can be solved for all $n$ and $s\in\{n-1,n+1\}$, there exists $\psi_{n+1}\in\Psi(n+1)$ and $\psi_{n-1}\in\Psi(n-1)$ and constants $\beta^+$, $\gamma^+$, $\delta^+$, $\varepsilon^+$, $\beta^-$, $\gamma^-$, $\delta^-$, $\varepsilon^-$ such that
	\begin{gather}
\psi_{n+1} =\big(\beta^+ L + \gamma^+ [K,L]_q + \delta^+ B + \varepsilon^+ D_1\big)\psi_n,\nonumber\\
			\psi_{n-1} =\big(\beta^-L + \gamma^- [K,L]_q + \delta^- B + \varepsilon^- D_1\big)\psi_n. \label{eq:systemLpsi}
	\end{gather}
	By substitution,\footnote{If $\lambda_{n-1}\neq\lambda_{n+1}$, the matrix $\left(\begin{smallmatrix}
			\beta^- & \gamma^- \\
			\beta^+ & \gamma^+
		\end{smallmatrix}\right)$ is non-singular and we can always solve \eqref{eq:systemLpsi} for $L\psi_n$ while eliminating $[K,L]_q$.}
	we can eliminate $[K,L]_q$ from \eqref{eq:systemLpsi} and obtain
	\begin{align*}
		L\psi_n = a_n \psi_{n-1} + b_n \psi_{n+1} + c_n B\psi_n + d_n D_1\psi_n,
	\end{align*}
	for suitable constants $a_n$, $b_n$, $c_n$, $d_n$. Since $B$ and $D_1$ commute with~$K$, it leaves its eigenspaces intact. Therefore,
	\[
	B\psi_n,D_1\psi_n \in \Psi(n),
	\]
	which proves (ii).
\end{proof}

\begin{Remark}
	In the original $\aw$, the eigenspaces of $K$ and $L$ were one-dimensional. However, this need not be true in general as we already saw in $\AW{2}$. There, the eigenspaces of $K_2$ and $L_2$ can have dimensions up to $(N_1+1)+(N_2+1)$.
\end{Remark}
We can use Lemma~\ref{lem:threetermeigenspace} twice for $L_2$, since it appears in two AW-relations: \eqref{eq:aw1(M2)} and~\eqref{eq:aw4}. Consequently, $L_2$ acts on $\psi_{n_1}^{n_2}\otimes \psi_{n_2}$ as a three term operator in the eigenspaces of both~$n_1$ and~$n_2$, which means $L_2$ is a 9-term operator.
\begin{Proposition}
	$L_2$ acts as a 9-term operator on $\psi_{n_1}^{n_2}\otimes \psi_{n_2}$. That is, there exists constants $a_{n_1,n_2}$, $b_{n_1,n_2}$, $c_{n_1,n_2}$, $d_{n_1,n_2}$ and $e_{n_1,n_2}$ such that
	\begin{gather}
L_2\big(\psi_{n_1}^{n_2}\otimes \psi_{n_2}\big) = a_{n_1,n_2}\ \psi_{n_1-1}^{n_2-1}\otimes \psi_{n_2-1}+b_{n_1,n_2}\ \psi_{n_1-1}^{n_2}\otimes \psi_{n_2}+c_{n_1,n_2}\ \psi_{n_1-1}^{n_2+1}\otimes \psi_{n_2+1}\nonumber\\
\hphantom{L_2\big(\psi_{n_1}^{n_2}\otimes \psi_{n_2}\big) =}{}
+ d_{n_1,n_2}\ \psi_{n_1}^{n_2-1}\otimes \psi_{n_2-1} + e_{n_1,n_2}\ \psi_{n_1}^{n_2}\otimes \psi_{n_2}+d_{n_1,n_2+1}\ \psi_{n_1}^{n_2+1}\otimes \psi_{n_2+1}\nonumber\\
\hphantom{L_2\big(\psi_{n_1}^{n_2}\otimes \psi_{n_2}\big) =}{}
+ c_{n_1+1,n_2-1}\ \psi_{n_1+1}^{n_2-1}\otimes \psi_{n_2-1}+b_{n_1+1,n_2}\ \psi_{n_1+1}^{n_2}\otimes \psi_{n_2}\nonumber\\
\hphantom{L_2\big(\psi_{n_1}^{n_2}\otimes \psi_{n_2}\big) =}{}
+ a_{n_1+1,n_2+1}\ \psi_{n_1+1}^{n_2+1}\otimes \psi_{n_2+1}, \label{eq:9terms}
	\end{gather}
	where
	\begin{gather}
n_2\in\left\{-\frac{N_2}{2},-\frac{N_2}{2}+1,\dots,\frac{N_2}{2}-1,\frac{N_2}{2}\right\},\nonumber\\
n_1\in\left\{n_2-\frac{N_1}{2},n_2-\frac{N_1}{2}+1,\dots,n_2+\frac{N_1}{2}-1,n_2+\frac{N_1}{2}\right\}.\label{eq:n1n2waarvandaan}
	\end{gather}
\end{Proposition}
\begin{proof}We apply Lemma \ref{lem:threetermeigenspace} twice and then we will see that $L_2$ has to be a 9-term operator. We will work with simultaneous eigenvectors of both $1\otimes K_1$ and $K_2$ given by
	\begin{gather*}
		\psi_{n_1}^{n_2}\otimes \psi_{n_2},
	\end{gather*}
	which corresponds, a bit confusingly, to either the eigenvalue $\lambda_{n_1}$ for $K_2$ or $\lambda_{n_2}$ for $K_1$. Moreover, $n_1$,~$n_2$~are as in \eqref{eq:n1n2waarvandaan}. Fix $n_1$ and $n_2$ and denote the eigenspace corresponding to $\lambda_{n_1}$ by $\Psi_{K_2}(n_1)$ and to $\lambda_{n_2}$ by $\Psi_{K_1}(n_2)$. First of all, since \eqref{eq:aw4} holds, Lemma~\ref{lem:threetermeigenspace} tells us that
	\begin{gather*}
		L_2 \Psi_{K_1}(n_2)\subset \Psi_{K_1}(n_2-1)\cup \Psi_{K_1}(n_2)\cup\Psi_{K_1}(n_2+1).
	\end{gather*}	
	Said differently, there exists $v_{1},v_2,v_3\in V_1$ such that
	\begin{gather}
		L_2 \big(\psi_{n_1}^{n_2}\otimes \psi_{n_2}\big)= v_1 \otimes \psi_{n_2-1}+v_2 \otimes \psi_{n_2}+v_3 \otimes \psi_{n_2+1}.\label{eq:L29termeerste3term}
	\end{gather}	
	Secondly, because \eqref{eq:aw1(M2)} holds, Lemma~\ref{lem:threetermeigenspace} tells us that
	\begin{gather*}
		L_2 \Psi_{K_2}(n_1)\subset \Psi_{K_2}(n_1-1)\cup \Psi_{K_2}(n_1)\cup\Psi_{K_2}(n_1+1).
	\end{gather*}	
	Combining this with \eqref{eq:L29termeerste3term} gives that
	\begin{gather*}
		v_1\otimes \psi_{n_2-1}\in \Psi_{K_2}(n_1-1)\cup \Psi_{K_2}(n_1)\cup\Psi_{K_2}(n_1+1),\\
		v_2\otimes \psi_{n_2} \in \Psi_{K_2}(n_1-1)\cup \Psi_{K_2}(n_1)\cup\Psi_{K_2}(n_1+1),\\
		v_3\otimes \psi_{n_2+1} \in \Psi_{K_2}(n_1-1)\cup \Psi_{K_2}(n_1)\cup\Psi_{K_2}(n_1+1).
	\end{gather*}
	Therefore, there exist constants $(c_k)_{k=0}^8$ such that
	\begin{gather*}
		v_1 = c_0 \psi_{n_1-1}^{n_2-1}\otimes \psi_{n_2-1} + c_1 \psi_{n_1}^{n_2-1}\otimes \psi_{n_2-1} +c_2 \psi_{n_1+1}^{n_2-1}\otimes \psi_{n_2-1},\\
		v_2 = c_3 \psi_{n_1-1}^{n_2}\otimes \psi_{n_2} + c_4 \psi_{n_1}^{n_2}\otimes \psi_{n_2} +c_5 \psi_{n_1+1}^{n_2}\otimes \psi_{n_2},\\
		v_3 = c_5 \psi_{n_1-1}^{n_2+1}\otimes \psi_{n_2+1} + c_7 \psi_{n_1}^{n_2+1}\otimes \psi_{n_2+1} +c_8 \psi_{n_1+1}^{n_2+1}\otimes \psi_{n_2+1}.
	\end{gather*}
	By Proposition \ref{prop:*reprexists}, we can require $L_2$ to be self-adjoint, which gives the form \eqref{eq:9terms}.
\end{proof}

To compute the coefficients of \eqref{eq:9terms}, we need to know how $M_2$ acts on $\psi_{n_1}^{n_2}\otimes \psi_{n_2}$. This is the subject of the next proposition, which is also of interest of its own. Since $M_2$ and $1\otimes K_1$ satisfy the AW-relations \eqref{eq:aw6} and $M_2$ commutes with $K_2$, we can show that $M_2$ acts as a three-term operator on the eigenvectors $\psi_{n_1}^{n_2}\otimes \psi_{n_2}$.
\begin{Proposition}\label{prop:casimir3term}
	We have
	\begin{gather*}
M_2\big(\psi_{n_1}^{n_2}\otimes \psi_{n_2}\big)= \tilde{a}_{n_2}\big(\bm{\alpha}_{M_2}^{n_1}\big)\psi_{n_1}^{n_2-1}\otimes \psi_{n_2-1}+\tilde{b}_{n_2}\big(\bm{\alpha}_{M_2}^{n_1}\big)\psi_{n_1}^{n_2}\otimes \psi_{n_2}\\
\hphantom{M_2\big(\psi_{n_1}^{n_2}\otimes \psi_{n_2}\big)=}{}
 +\tilde{a}_{n_2+1}\big(\bm{\alpha}_{M_2}^{n_1}\big)\psi_{n_1}^{n_2+1}\otimes \psi_{n_2+1}, 
	\end{gather*}
	where $\bm{\alpha}_{M_2}^{n_1}=(\alpha_0,-N_1-1,-\alpha_0-2n_1,-N_2-1)$. The coefficients $\tilde{a}_n^2$ and $\tilde{b}_n$ can be found from~\eqref{eq:annewtilde} and~\eqref{eq:bnnewtilde}, respectively.
\end{Proposition}
\begin{proof}
	Denote by $\Psi_{K_1}(n_2)$ the eigenspace of $1\otimes K_1$ corresponding to the eigenvalue $\lambda_{n_2}$. From Lemma \ref{lem:threetermeigenspace}(ii) and the AW-relations \eqref{eq:aw6}, we know that
	\begin{gather}
		M_2 \big(\psi_{n_1}^{n_2}\otimes \psi_{n_2}\big) \in \Psi({n_2}-1)\cup\Psi(n_2)\cup\Psi (n_2+1).\label{eq:M2drietermeigenspace}
	\end{gather}
	Since $M_2$ and $K_2$ commute, the eigenspaces of $K_2$, denoted by $\Psi_{K_2}(n_1)$ for the eigenvalue $\lambda_{n_1}$, are invariant under the action of $M_2$,
	\begin{gather*}
		M_2 \Psi_{K_2}(n_1) \subset \Psi_{K_2}(n_1).
	\end{gather*}
	Combining this with \eqref{eq:M2drietermeigenspace} implies that there exists constants $c_{n_2-1}$, $c_{n_2}$ and $c_{n_2+1}$ such that
	\begin{gather*}
		M_2 \big(\psi_{n_1}^{n_2}\otimes \psi_{n_2}\big) = c_{n_2-1}\psi_{n_1}^{n_2-1}\otimes \psi_{n_2-1}+c_{n_2}\psi_{n_1}^{n_2}\otimes \psi_{n_2}+c_{n_2+1}\psi_{n_1}^{n_2+1}\otimes \psi_{n_2+1}.
	\end{gather*}
	It remains to find the constants $c_{n_2-1}$, $c_{n_2}$ and $c_{n_2+1}$. Observe that we can now interpret~$K_2$ in the AW-relations \eqref{eq:aw6} as the `constant' $\lambda_{n_1}$. Therefore, we can repeat the proof of Proposition~\ref{prop:opnieuw3term}. Note that we are in the setting $C_0>0$ and $C_1 <0$. Therefore, we have to use the formulas~\eqref{eq:annewtilde} and~\eqref{eq:bnnewtilde}. We find
	\[
	c_{n_2-1}=\tilde{a}_{n_2}\big(\bm{\alpha}_{M_2}^{n_1}\big),\qquad c_{n_2}=\tilde{b}_{n_2}\big(\bm{\alpha}_{M_2}^{n_1}\big)\qquad \text{and}\qquad c_{n_2+1}=\tilde{a}_{n_2+1}\big(\bm{\alpha}_{M_2}^{n_1}\big).\tag*{\qed}
	\]\renewcommand{\qed}{}
\end{proof}

Let us now compute $L_2$ explicitly.
\begin{Theorem}The coefficients in \eqref{eq:9terms} of the 9-term operator $L_2$ are given by
	\begin{gather*}
		a_{n_1,n_2} = \frac{a_{n_1}\big(\bm{\alpha}_{L_1}^{n_2}\big)d_{n_1-1,n_2}-a_{n_1}\big(\bm{\alpha}_{L_1}^{n_2-1}\big)d_{n_1,n_2}} {b_{n_1-1}\big(\bm{\alpha}_{L_1}^{n_2-1}\big)-b_{n_1}\big(\bm{\alpha}_{L_1}^{n_2}\big)} ,\\ 
		b_{n_1,n_2} = \frac{a_{n_1}\big(\bm{\alpha}_{L_1}^{n_2}\big)\big(A_0\lambda_{n_2} +\chq(1)A_3^{(2)}\big)}{\chq(2n_2+\alpha_0-1)\chq(2n_2+\alpha_0+1)} ,\\
		c_{n_1,n_2} = \frac{a_{n_1}\big(\bm{\alpha}_{L_1}^{n_2}\big)d_{n_1-1,n_2+1}-a_{n_1}\big(\bm{\alpha}_{L_1}^{n_2+1}\big)d_{n_1,n_2+1}}{b_{n_1-1}\big(\bm{\alpha}_{L_1}^{n_2+1}\big) -b_{n_1}\big(\bm{\alpha}_{L_1}^{n_2}\big)} ,\\
d_{n_1,n_2} =\frac{\tilde{a}_{n_2}\big(\bm{\alpha}_{M_2}^{n_1}\big) (\chq(1)A_1-A_2\lambda_{n_1} )}{\chq(2n_1+\alpha_0-1)\chq(2n_1+\alpha_0+1)} ,\\
e_{n_1,n_2} =\frac{b_{n_2}\big(\bm{\alpha}_{M_2}^{n_1}\big) (\chq(1)A_1-A_2\lambda_{n_1})+A_0(A_1\lambda_{n_1}+\chq(1)A_2 )}{\chq(2n_1+\alpha_0-1)\chq(2n_1+\alpha_0+1)} .
	\end{gather*}
\end{Theorem}
\begin{proof}
	Computing $b_{n_1,n_2}$, $d_{n_1,n_2}$, $e_{n_1,n_2}$ is similar to calculating $b_n$ in the original~$\aw$. The `corner' terms $a_{n_1,n_2}$, $c_{n_1,n_2}$ can be deduced from the commutativity of $L_2$ and $L_1\otimes 1$ and the expressions for $b_{n_1,n_2}$, $d_{n_1,n_2}$, $e_{n_1,n_2}$. Also note that we do not need to worry about consistency of the coefficients because of Proposition~\ref{prop:*reprexists}.
	
Let us start with $d_{n_1,n_2}$ and $e_{n_1,n_2}$. We can apply both sides of the first AW-relation of \eqref{eq:aw1(M2)} to $\psi_{n_1}^{n_2}\otimes \psi_{n_2}$,
	\begin{gather*}
\big(\chq(2) K_2L_2K_2 - K_2^2L_2 -L_2K_2^2\big) \psi_{n_1}^{n_2}\otimes \psi_{n_2}\\
		\qquad{} = \big[(\shq(1))^2(A_0A_1 -A_2M_2)K_2 -(\shq(2))^2 L_2 \\
		\qquad\quad{} + \chq(1)(\shq(1))^2(A_1M_2+A_0A_2)\big]\psi_{n_1}^{n_2}\otimes \psi_{n_2}.
	\end{gather*}
	If we use Proposition \ref{prop:casimir3term} and
	\[
	K_2\psi_{n_1}^{n_2}\otimes \psi_{n_2}=\lambda_{n_1}\psi_{n_1}^{n_2}\otimes \psi_{n_2},
	\]
	we get 9 equations. Each one coming from the term in front of one of the eigenvectors $\psi_{n_1+i}^{n_2+j}\otimes \psi_{n_2+j}$ with $i,j\in\{-1,0,1\}$. Similar to the original $\aw$, one can show that the $6$ equations coming from $\psi_{n_1\pm 1}^{n_2+j}\otimes \psi_{n_2+j}$ are satisfied automatically using \eqref{eq:spectrumrecursiverelations}. Looking at the terms in front of $\psi_{n_1}^{n_2+j}\otimes \psi_{n_2+j}$ for $j\in\{-1,0,1\}$ gives
	\begin{gather}
		c_{n_1}d_{n_1,n_2} = (\shq(1))^2\big(\chq(1)A_1\tilde{a}_{n_2}(\bm{\alpha}_{M_2}^{n_1}-A_2\tilde{a}_{n_2}\big(\bm{\alpha}_{M_2}^{n_1}\big)\lambda_{n_1} )\big)\label{eq:dn1n2},\\
		c_{n_1}e_{n_1,n_2} = (\shq(1))^2\big(\big(A_0A_1 -A_2\tilde{b}_{n_2}\big(\bm{\alpha}_{M_2}^{n_1}\big)\big)\lambda_{n_1} \! +\chq(1)(A_1\tilde{b}_{n_2}\big(\bm{\alpha}_{M_2}^{n_1}\big)\!+A_0A_2)\big),\label{eq:en1n2}\!\!\!\!\\
		c_{n_1}d_{n_1,n_2+1} = (\shq(1))^2\big(\chq(1)A_1\tilde{a}_{n_2+1}\big(\bm{\alpha}_{M_2}^{n_1}-A_2\tilde{a}_{n_2+1}\big(\bm{\alpha}_{M_2}^{n_1}\big)\lambda_{n_1} \big)\big),\label{eq:dn1n2+1}
	\end{gather}
	where
	\begin{align*}
c_{n_1}& =\lambda_{n_1}^2\left(\chq(2)-2\right)+\chq(2)^2 \\
& = (\shq(1))^2\chq(2n_1+\alpha_0-1)\chq(2n_1+\alpha_0+1),
	\end{align*}
	using \eqref{eq:spectrumrecursiverelations} and \eqref{eq:difference eigenvalues}. It is easy to see that \eqref{eq:dn1n2+1} is just \eqref{eq:dn1n2} with $n_2+1$ instead of $n_2$. The coefficients $d_{n_1,n_2}$ and $e_{n_1,n_2}$ can be computed from \eqref{eq:dn1n2} and \eqref{eq:en1n2},
	\begin{gather}
		d_{n_1,n_2} = \frac{\tilde{a}_{n_2}\big(\bm{\alpha}_{M_2}^{n_1}\big) (\chq(1)A_1-A_2\lambda_{n_1})}{\chq(2n_1+\alpha_0-1)\chq(2n_1+\alpha_0+1)},\nonumber\\
		e_{n_1,n_2} = \frac{\tilde{b}_{n_2}\big(\bm{\alpha}_{M_2}^{n_1}\big) (\chq(1)A_1-A_2\lambda_{n_1} )+A_0 (A_1\lambda_{n_1}+\chq(1)A_2 )}{\chq(2n_1+\alpha_0-1)\chq(2n_1+\alpha_0+1)}.\label{eq:en1n2M2}
	\end{gather}
	Similarly, we can apply both sides of the first AW-relation of $\eqref{eq:aw4}$ to $\psi_{n_1}^{n_2}\otimes \psi_{n_2}$ to compute $b_{n_1,n_2}$ and $e_{n_1,n_2}$. Instead of Proposition \ref{prop:casimir3term}, we will use Proposition \ref{prop:opnieuw3term}. Equating the terms in front of each of the eigenvectors $\psi_{n_1+i}^{n_2+j}\otimes \psi_{n_2+j}$ with $i,j\in\{-1,0,1\}$ again gives $9$ equations, where now the $6$ equations coming from $\psi_{n_1+ i}^{n_2\pm 1}\otimes \psi_{n_2\pm 1}$, $i\in\{-1,0,1\}$, are satisfied automatically. The other 3 lead in a similar way as before to
	\begin{gather}
		b_{n_1,n_2} = \frac{a_{n_1}\big(\bm{\alpha}_{L_1}^{n_2}\big)\big(A_0\lambda_{n_2} +\chq(1)A_3^{(2)}\big)}{\chq(2n_2+\alpha_0-1)\chq(2n_2+\alpha_0+1)},\nonumber\\
		e_{n_1,n_2} = \frac{b_{n_1}\big(\bm{\alpha}_{L_1}^{n_2}\big)\big(A_0\lambda_{n_2} +\chq(1)A_3^{(2)}\big)-A_2\big(A_3^{(2)}\lambda_{n_2}-\chq(1)A_0\big)}{\chq(2n_2+\alpha_0-1)\chq(2n_2+\alpha_0+1)}.\label{eq:en1n2L1}
	\end{gather}
By Proposition \ref{prop:*reprexists}, we know that both expressions \eqref{eq:en1n2M2} and \eqref{eq:en1n2L1} are consistent. One can also check this by doing a simple, quite tedious computation.
	
Let us now compute the corner terms $a_{n_1,n_2}$ and $c_{n_1,n_2}$. Since $L_2$ and $L_1\otimes 1$ commute, we have
	\begin{gather*}
		L_2 (L_1\otimes 1) \psi_{n_1}^{n_2}\otimes \psi_{n_2} = (L_1\otimes 1)L_2 \psi_{n_1}^{n_2}\otimes \psi_{n_2}.
	\end{gather*}
	Working this out gives 15 equations for each of the eigenvectors $\psi_{n_1+i}^{n_2+j}\otimes\psi_{n_2+j}$, where $i\in\{-2,-1,0,1,2\}$ and $j\in\{-1,0,1\}$. Let us first focus on the eigenvector $\psi_{n_1-1}^{n_2-1}\otimes\psi_{n_2-1}$ for which there are two possibilities to get terms in front. Comparing these terms on both sides gives
	\begin{gather*}
		a_{n_1,n_2}b_{n_1}\big(\bm{\alpha}_{L_1}^{n_2}\big)+d_{n_1-1,n_2}a_{n_1}\big(\bm{\alpha}_{L_1}^{n_2}\big)= b_{n_1-1}\big(\bm{\alpha}_{L_1}^{n_2}\big)a_{n_1,n_2}+a_{n_1}\big(\bm{\alpha}_{L_1}^{n_2-1}\big)d_{n_1,n_2}.
	\end{gather*}
	This is equivalent to
	\begin{gather*}
		a_{n_1,n_2} = \frac{a_{n_1}\big(\bm{\alpha}_{L_1}^{n_2}\big)d_{n_1-1,n_2}-a_{n_1}\big(\bm{\alpha}_{L_1}^{n_2-1}\big)d_{n_1,n_2}}{b_{n_1-1}\big(\bm{\alpha}_{L_1}^{n_2-1}\big) -b_{n_1}\big(\bm{\alpha}_{L_1}^{n_2}\big)}.
	\end{gather*}
	Doing the same for $\psi_{n_1-1}^{n_2+1}\otimes\psi_{n_2+1}$ gives the expression for $c_{n_1,n_2}$.
\end{proof}

Since $L_2$ acts as a $9$-term operator on the basis $\psi_{n_1}^{n_2} \otimes \psi_{n_2}$ and as a multiplication operator on the basis $\phi_{m_1} \otimes \phi_{m_2}^{m_1}$, it follows immediately that the overlap coefficient $P_{n_1,n_2}(m_1,m_2)$ defined by~\eqref{eq:defbivarqracah}
satisfies a $9$-term recurrence relation, or equivalently, it is an eigenfunction of a $9$-term difference operator in the variables $n_1$ and $n_2$. Similarly, from the action of $L_1\otimes1$ it follows that $P_{n_1,n_2}(m_1,m_2)$ is also an eigenfunction of a $3$-term difference operator. In this way we recover Iliev's difference operators for the bivariate $q$-Racah polynomials, see \cite[Proposition~4.5 and Remark~2.3]{Iliev2011} in case $d=2$. So $\AW{2}$ encodes the bispectral properties of Gasper and Rahman's (or Tratnik-type) bivariate $q$-Racah polynomials.

\begin{Remark}
	A recent paper \cite{CFR} studies the rank 2 Racah algebra, which can be considered as a~$q=1$ version of a~rank~2 Askey--Wilson algebra. The Tratnik-type bivariate Racah polynomials appear as overlap coefficients, but also other overlap coefficients resembling $9j$-symbols are studied. These are shown to be `Griffith-like' bivariate Racah polynomials, which are different from the bivariate Tratnik-type Racah polynomials, and the algebraic setting provides difference operators for these polynomials. This shows that that the rank 2 Racah algebra not only encodes the spectral properties of the Tratnik-type bivariate Racah polynomials, but also of other bivariate functions. It would be very interesting to see if similar `Griffith-like' bivariate $q$-Racah polynomials can be obtained in the setting of the algebra $\AW{2}$. Presumably this requires the extra relations \eqref{eq:aw4relationswithholes} that are present in $\AW{2}$ inside $\Uq\otimes\Uq$, but not in the general setup of Definition \ref{Def:AW2}.
\end{Remark}

\appendix

\section{Proof of Theorem \ref{thm:tabel}}\label{app:uquqbivaraw}
Let $Y_K$, $Y_L$, $\Omega$ as in Section~\ref{sec:uq} and $B$, $C_0$, $C_1$, $D_0$, $D_1$ as in Theorem~\ref{thm:tpeaw3}. We will show that each pair of the elements $\yzero, \dyzero, \yone, \dyone, \Delta(\Omega)\in\Uq\otimes\Uq$ either commutes or satisfies the AW-relations~\eqref{eq:awrelation2} and~\eqref{eq:awrelation1}. The pairs of non-commuting elements satisfy these relations with the structure parameters in
Table~\ref{tab:tpegenapp}, where $\theta=-\shq(1)^{-2}(a_{\hat{E}}b_{\hat{F}})$.

\begin{table}[h]\centering \renewcommand{\arraystretch}{1.2}\small
	\begin{tabular}{|l|l|l|l|l|l|l|l|}
		\hline
		Generator 1 & Generator 2 & $A_0$ & $A_1$ & $A_2$ & $A_3$ & $A_4$ & $A_5$ \\ \hline
		$\dyzero$ & $\dyone$ & $a_s$ & $b_t$ & $\theta$ & $\Delta\left(\Omega\right)$ & $b_{\hat{E}}b_{\hat{F}}$ & $a_{\hat{E}}a_{\hat{F}}$ \\ \hline
		$\yzero$ & $\dyone$ & $a_s$ & $\yone$ & $\theta$ & $1\otimes\Omega$ & $b_{\hat{E}}b_{\hat{F}}$ & $a_{\hat{E}}a_{\hat{F}}$ \\ \hline
		$\dyzero$ & $\yone$ & $\yzero$ & $b_t$ & $\theta$ & $\Omega\otimes 1$ & $b_{\hat{E}}b_{\hat{F}}$ & $a_{\hat{E}}a_{\hat{F}}$ \\ \hline
		$\yzero$ & $\Delta\left(\Omega\right)$ & $a_s$ & $\Omega\otimes 1$ & $-\dyzero$ & $1\otimes\Omega$ & $-(\shq(1))^{2}$ &$a_{\hat{E}}a_{\hat{F}}$ \\ \hline
		$\Delta\left(\Omega\right)$ & $\yone$ & $1\otimes\Omega$ & $b_t$ & $-\dyone$ & $\Omega\otimes 1$ & $ b_{\hat{E}}b_{\hat{F}}$ &$-(\shq(1))^{2}$ \\ \hline
	\end{tabular}
	\caption{Askey--Wilson algebra relations of twisted primitive elements.}\label{tab:tpegenapp}
\end{table}

The commuting part of the proof was already discussed in Section \ref{sec:uq} as well as the first row of Table \ref{tab:tpegenapp}. We will show rows three and four. The others follow by symmetry. For elements $a$, $b$, $c$ in an algebra $A$, define
\[
f_A(a,b,c):=\chq(2)abc-bca-cab,
\]
coming from the $\aw$ relations \eqref{eq:awrelation2} and \eqref{eq:awrelation1}. That is,
\begin{gather*}
	f_{\Uq}(Y_L,Y_K,Y_L)=BY_L + C_0Y_K + D_0,\\
	f_{\Uq}(Y_K,Y_L,Y_K)=BY_K + C_1 Y_L + D_1.
\end{gather*}
Let us start with showing the AW-relation between $\Delta(Y_K)$ and $Y_L\otimes 1$. Explicitly calculating gives
\begin{gather*}
	f_{\Uq\tensor{2}}(\Delta(Y_K),Y_L\otimes 1,\Delta(Y_K)) =f_{\Uq}\big(\hat{K}^2,Y_L,\hat{K}^2\big)\otimes Y_K^2 +f\big(\tilde{Y_K},Y_L,\tilde{Y_K}\big)\otimes 1 \\
\hphantom{f_{\Uq\tensor{2}}(\Delta(Y_K),Y_L\otimes 1,\Delta(Y_K)) =}{}
 + \big[f_{\Uq}\big(\hat{K}^2,Y_L,\tilde{Y_K}\big)+f_{\Uq}\big(\tilde{Y_K},Y_L,\hat{K}^2\big)\big]\otimes Y_K.
\end{gather*}
Using~\eqref{eq:embeddingconstants} with $a_{\hat{E}}=a_{\hat{F}}=0$, we obtain
\begin{gather}
	f_{\Uq\tensor{2}}(\Delta(Y_K),Y_L\otimes 1,\Delta(Y_K)) =(\shq(1))^2b_t \hat{K}^2\otimes Y_K^2 +(\tilde{B}\tilde{Y_K}+C_1 Y_L +\tilde{D})\otimes 1 \nonumber \\
\hphantom{f_{\Uq\tensor{2}}(\Delta(Y_K),Y_L\otimes 1,\Delta(Y_K)) =}{}
 + \big[f_{\Uq}\big(\hat{K}^2,Y_L,\tilde{Y_K}\big)+f_{\Uq}\big(\tilde{Y_K},Y_L,\hat{K}^2\big)\big]\otimes Y_K,\label{eq:DeltaK0eerst}
\end{gather}
where $\tilde{B}=\big(a_{\hat{E}}b_{\hat{F}}+a_{\hat{F}}b_{\hat{E}}\big)\Omega$ and $\tilde{D_1}= \chq(1)a_{\hat{E}}a_{\hat{F}}b_t\Omega$.
We have
\begin{gather*}
	f_{\Uq}\big(\hat{K}^2,Y_L,\tilde{Y_K}\big)+f_{\Uq}\big(\tilde{Y_K},Y_L,\hat{K}^2\big)\\
\qquad{} = f\big(\tilde{Y_K}+\hat{K}^2,Y_L,\tilde{Y_K}+\hat{K}^2\big) -f\big(\tilde{Y_K},Y_L,\tilde{Y_K}\big)-f\big(\hat{K}^2,Y_L,\hat{K}^2\big)\\
	\qquad{}= (\shq(1))^2 b_t \tilde{Y_K}+\tilde{B} \hat{K}^2 +\tilde{D},
\end{gather*}
where $\tilde{D}=-\chq(1)\big(a_{\hat{E}}b_{\hat{F}}+a_{\hat{F}}b_{\hat{E}}\big)$.
Therefore, \eqref{eq:DeltaK0eerst} becomes
\begin{gather*}
f_{\Uq\tensor{2}}(\Delta(Y_K),Y_L\otimes 1,\Delta(Y_K))\\
\qquad{} =\big(\tilde{B}\otimes 1\big)\Delta(Y_K) + C_1 Y_L\otimes 1 + \tilde{D_1}
 + (\shq(1))^2b_t \Delta(Y_K)(1\otimes Y_K) +\tilde{D}(1\otimes Y_K). 
\end{gather*}
For the other relation of $Y_L\otimes 1$ and $\Delta(Y_K)$, we obtain
\begin{gather*}
	f_{\Uq\tensor{2}}(Y_L\otimes 1,\Delta(Y_K),Y_L\otimes 1) = f_{\Uq}(Y_L,\hat{K}^2,Y_L)\otimes Y_K + f_{\Uq}(Y_L,\tilde{Y_K},Y_L)\otimes 1\\
\hphantom{f_{\Uq\tensor{2}}(Y_L\otimes 1,\Delta(Y_K),Y_L\otimes 1)}{}
	= ((\shq(1))^2b_t Y_L+C_0\hat{K}^2+\tilde{D_0})\otimes Y_K \\
\hphantom{f_{\Uq\tensor{2}}(Y_L\otimes 1,\Delta(Y_K),Y_L\otimes 1)=}{}
 + \big(\tilde{B}Y_L+C_0\tilde{Y_K}+\tilde{D}b_t\big)\otimes 1\\
\hphantom{f_{\Uq\tensor{2}}(Y_L\otimes 1,\Delta(Y_K),Y_L\otimes 1)}{}
	= \big(\tilde{B}\otimes 1\big)(Y_L\otimes 1)+ C_0 \Delta(Y_K)+\tilde{D}b_t \\
\hphantom{f_{\Uq\tensor{2}}(Y_L\otimes 1,\Delta(Y_K),Y_L\otimes 1)=}{}
	+(\shq(1))^2b_t(Y_L\otimes 1)(1\otimes Y_K) + (\tilde{D_0}\otimes 1)(1\otimes Y_K),
\end{gather*}
where $\tilde{D_0}=\chq(1)b_{\hat{E}}b_{\hat{F}}\Omega$.

Let us now show the AW-relation between $\Delta(\Omega)$ and $1\otimes Y_K$. Let us first calculate the coproduct of the Casimir of $\Uq$. Using that $\hat{K}$, $\hat{K}^{-1}$ are group-like\footnote{An element $Y\in\Uq$ is called group-like if $\Delta(Y)=Y\otimes Y$.} and ${\hat{E}}$, ${\hat{F}}$ are twisted primitive with respect to $\hat{K}$, we obtain
\begin{gather*}
\frac{\Delta(\Omega)}{(\shq(1))^2}=\frac{q^{-1} \hat{K}^2\otimes \hat{K}^2 +q \hat{K}^{-2}\otimes \hat{K}^{-2}}{(\shq(1))^2} + \big(\hat{K}\otimes {\hat{E}} + {\hat{E}}\otimes \hat{K}^{-1}\big)\big(\hat{K}\otimes {\hat{F}} + {\hat{F}}\otimes \hat{K}^{-1}\big)\nonumber\\
\hphantom{\frac{\Delta(\Omega)}{(\shq(1))^2}}{}
	=\frac{\hat{K}^2\otimes \Omega + \Omega\otimes \hat{K}^{-2}\! - \chq(1) \hat{K}^2\otimes \hat{K}^{-2}}{(\shq(1))^2} + \big(\hat{K}{\hat{F}}\otimes {\hat{E}}\hat{K}^{-1}\! + {\hat{E}}\hat{K}\otimes \hat{K}^{-1}{\hat{F}}\big).
\end{gather*}
Note that all the factors that appear are specific versions of $Y_K$ and $Y_L$. For example, ${\hat{E}}\hat{K}^{-1}$ is~$Y_L$ where $b_{\hat{F}}=b_t=0$ and $b_{\hat{E}}=q^{1/2}$. Therefore, we can use~\eqref{eq:embeddingconstants} to obtain
\begin{gather}
	f_{\Uq^{\otimes2}}(1\otimes Y_K,\Delta(\Omega),1\otimes Y_K)= (\shq(1))^2 \hat{K}^2\otimes Y_K^2\Omega + \Omega \otimes f_{\Uq}\big(Y_K,\hat{K}^{-2},Y_K\big)\nonumber \\
\hphantom{f_{\Uq^{\otimes2}}(1\otimes Y_K,\Delta(\Omega),1\otimes Y_K)=}{}
+ (\shq(1))^2\hat{K}{\hat{F}}\otimes f_{\Uq}\big(Y_K,{\hat{E}}\hat{K}^{-1},Y_K\big)\nonumber	\\
\hphantom{f_{\Uq^{\otimes2}}(1\otimes Y_K,\Delta(\Omega),1\otimes Y_K)=}{}
+(\shq(1))^2{\hat{E}}\hat{K}\otimes f_{\Uq}\big(Y_K,\hat{K}^{-1}{\hat{F}},Y_K\big)\nonumber \\
\hphantom{f_{\Uq^{\otimes2}}(1\otimes Y_K,\Delta(\Omega),1\otimes Y_K)=}{}
- \chq(1) \hat{K}^2\otimes f_{\Uq}\big(Y_K,\hat{K}^{-2},Y_K\big)\label{eq:k0coproductcasimir}.
\end{gather}
Using \eqref{eq:awrelation2} and \eqref{eq:embeddingconstants}, we get
\begin{gather*}
	f_{\Uq}\big(Y_K,\hat{K}^{-2},Y_K\big) = (\shq(1))^2 a_s Y_K + C_1\hat{K}^{-2} + \chq(1)a_{\hat{E}}a_{\hat{F}}\Omega,\\
	f_{\Uq}\big(Y_K,{\hat{E}}\hat{K}^{-1},Y_K\big) = q^{1/2}a_{\hat{F}}\Omega Y_K + C_1 {\hat{E}}\hat{K}^{-1} - q^{1/2}\chq(1)a_{\hat{F}}a_s,\\
	f_{\Uq}\big(Y_K,\hat{K}^{-1}{\hat{F}},Y_K\big) = q^{1/2}a_{\hat{E}}\Omega Y_K + C_1 \hat{K}^{-1}{\hat{F}} - q^{1/2}\chq(1)a_{\hat{E}}a_s.
\end{gather*}
Therefore, \eqref{eq:k0coproductcasimir} becomes
\begin{gather*} f_{\Uq\tensor{2}}(1\otimes Y_K,\Delta(\Omega),1\otimes Y_K)=(\shq(1))^2(1\otimes \Omega)\Delta(Y_K)(1\otimes Y_K) + C_1\Delta(\Omega) \\
\hphantom{f_{\Uq\tensor{2}}(1\otimes Y_K,\Delta(\Omega),1\otimes Y_K)=}{}
		+ \chq(1)a_{\hat{E}}a_{\hat{F}} \Omega\otimes\Omega + (\shq(1))^2a_s(\Omega\otimes 1)(1\otimes Y_K) \\
\hphantom{f_{\Uq\tensor{2}}(1\otimes Y_K,\Delta(\Omega),1\otimes Y_K)=}{}
		- \chq(1)(\shq(1))^2a_s \Delta(Y_K).
	\end{gather*}
Computing $f_{\Uq\tensor{2}}(\Delta(\Omega),1\otimes Y_K,\Delta(\Omega))$ is the most tedious, since we end up with $75$ terms. We will not present the full calculation, as it is just a straightforward computation. One can simplify the calculation by observing that some terms vanish because
\begin{gather*}
	f_A(a,b,c)=0\qquad \text{if }ab=q^2ba \ \text{ and }\ cb=q^2bc,\\
	f_A(a,b,c)+f_A(c,b,a)=0\qquad \text{if }ab=q^2ba, ac=q^{-4}ca \ \text{ and }\ bc=cb.
\end{gather*}

\section{Calculations for Remark \ref{rem:correspondenceaw4aw2}}\label{app:calculationsrelationsAW4}
Let $S\subset \Uq\otimes\Uq$ be the subalgebra generated by
\[Y_L\otimes 1,\quad \Delta(Y_L),\quad 1\otimes Y_K,\quad \Delta(Y_K),\quad \Omega\otimes 1,\quad 1\otimes \Omega,\quad \Delta(\Omega).\]
Then we have the following correspondence between $S$ and $\mathrm{AW}(4)$:
\begin{gather}
		\Lambda_{\{1\}} =b_t, \qquad \Lambda_{\{2\}}=\Omega\otimes 1, \qquad \Lambda_{\{3\}}=1 \otimes \Omega, \qquad \Lambda_{\{4\}}=a_s, \nonumber\\
		\Lambda_{\{1,2\}} =Y_L\otimes 1,\qquad \Lambda_{\{2,3\}}=\Delta(\Omega), \qquad \Lambda_{\{3,4\}}=1\otimes Y_K, \nonumber\\
		\Lambda_{\{1,2,3\}} =\Delta(Y_L), \qquad \Lambda_{\{1,2,3,4\}}=-\frac{a_{\hat{E}}b_{\hat{F}}+a_{\hat{F}}b_{\hat{E}}}{(\shq(1))^2}, \qquad \Lambda_{\{2,3,4\}} =\Delta(Y_K).\label{eq:StoAW4}
	\end{gather}
We can write the relation \cite[equation~(23)]{DeBie2020}, when translating to the Askey--Wilson algebra setting, as
\[
\Lambda_{C}=\alpha[\Lambda_{A},\Lambda_{B}]_q + \beta\big(\Lambda_{A\cap B}\Lambda_{A\cup B}+\Lambda_{A\backslash B}\Lambda_{B\backslash A}\big),
\]
with $\alpha= -(\shq(2))^{-1}$ and $\beta=(\chq(1))^{-1}$. We use this relation to define $\Lambda_{\{1,3\}}$ by
\[
\Lambda_{\{1,3\}}=\alpha[\Lambda_{\{1,2\}},\Lambda_{\{2,3\}}]_q + \beta \left( \Lambda_{\{2\}}\Lambda_{\{1,2,3\}}+\Lambda_{\{1\}}\Lambda_{\{3\}} \right).
\]
Similarly, we define $\Lambda_{\{1,4\}}$ and $\Lambda_{\{1,3,4\}}$ by
\begin{gather*}
	\Lambda_{\{1,4\}}=\alpha \big[\Lambda_{\{1,2,3\}},\Lambda_{\{2,3,4\}} \big]_q + \beta \big( \Lambda_{\{2,3\}}\Lambda_{\{1,2,3,4\}}+\Lambda_{\{1\}}\Lambda_{\{4\}} \big),\\
	\Lambda_{\{1,3,4\}}=\alpha \big[\Lambda_{\{1,2\}},\Lambda_{\{2,3,4\}} \big]_q + \beta \big( \Lambda_{\{2\}}\Lambda_{\{1,2,3,4\}}+\Lambda_{\{1\}}\Lambda_{\{3,4\}} \big).
\end{gather*}
Let us now check that the relation
\begin{gather}
	\Lambda_{\{1,4\}}=\alpha \big[\Lambda_{\{1,3\}},\Lambda_{\{3,4\}} \big]_q + \beta \big( \Lambda_{\{3\}}\Lambda_{\{1,3,4\}}+\Lambda_{\{1\}}\Lambda_{\{4\}} \big)\label{eq:relationtocheck}
\end{gather}
also holds in $S$. We will substitute \eqref{eq:StoAW4} in \eqref{eq:relationtocheck} and focus only on factors of the form $\hat{F}^2\otimes X$, with $X\in \Uq$. Substituting the definition of $\Lambda_{\{1,4\}}$, we obtain for the left-hand side,
\[
\alpha[\Delta(Y_L),\Delta(Y_K)]_q - \beta \frac{(a_{\hat{E}}b_{\hat{F}}+a_{\hat{F}}b_{\hat{E}})\Delta(\Omega)}{(\shq(1))^2} +\beta b_ta_s.
\]
The only term that has a factor $\hat{F}^2\otimes X$ is $\alpha[\Delta(Y_L),\Delta(Y_K)]_q$, which is equal to
\[
-a_{\hat{F}}b_{\hat{F}}\ 	\hat{F}^2\otimes \hat{K}^{-2}.
\]
For the right-hand side of \eqref{eq:relationtocheck}, we substitute the definitions for $\Lambda_{\{1,3\}}$ and $\Lambda_{\{1,3,4\}}$ to obtain
\begin{gather*}
	\alpha \big[\alpha[Y_L\otimes 1,\Delta(\Omega)]_q + \beta (\Omega\otimes 1)\Delta(Y_L)+\beta b_t(1\otimes \Omega),1\otimes Y_K \big]_q \\
	\qquad+\beta (1\otimes\Omega)\left(\alpha[Y_L\otimes 1,\Delta(Y_K)]_q - \beta \frac{(a_{\hat E}b_{\hat F}+ a_{\hat F}b_{\hat E})(\Omega\otimes1)}{(\shq(1))^2}+\beta b_t1\otimes Y_K\right)+ \beta b_ta_s.
\end{gather*}
Here, the only terms that have a factor $\hat{F}^2\otimes X$ is $\alpha^2 [[Y_L\otimes 1,\Delta(\Omega)]_q,1\otimes Y_K]_q$ and\\ $\alpha \beta (1\otimes\Omega)[Y_L\otimes 1,\Delta(Y_K)]_q$. These are, respectively,
\[
\alpha^2 a_{\hat{F}}b_{\hat{F}} \shq(2) (\shq(1))^2 \hat{F}^2\otimes [E,F]_q \qquad \text{and} \qquad \alpha \beta a_{\hat{F}}b_{\hat{F}} \shq(2) \hat{F}^2\otimes \Omega.
\]
It follows from writing out $\Omega$ and using defining relations for $\Uq$ that the factor with $\hat F^2 \otimes X$ on the right-hand side of relation~\eqref{eq:relationtocheck} is equal to the the factor on the left-hand side. In the same way the other terms in~\eqref{eq:relationtocheck} can be checked, showing that this relation is indeed valid in~$S$.

\subsection*{Acknowledgements}
We thank the referees for their valuable suggestions which helped to improve the paper.

\pdfbookmark[1]{References}{ref}
\LastPageEnding

\end{document}